\documentclass[10pt,
]{amsart}
\usepackage{%amsmath,
amssymb, graphicx,datetime
}
\newtheorem{thm}{Theorem}
\newtheorem{lemma}{Lemma}

\theoremstyle{remark}
\newtheorem{example}{Example}
\newtheorem{remark}{Remark}
 
\usepackage{enumerate}

\DeclareMathOperator{\supp}{supp}

\DeclareMathOperator{\WF}{WF}

\newcommand{\eps}{\varepsilon}
 
\newcommand{\R}{{\bf R}} 
\newcommand{\Id}{\mbox{Id}} 
\renewcommand{\r}[1]{(\ref{#1})} 
\newcommand{\PDO}{$\Psi$DO} 
\newcommand{\be}[1]{\begin{equation}\label{#1}} 
\newcommand{\ee}{\end{equation}} 

\renewcommand{\d}{\mathrm{d}}

\renewcommand{\i}{\mathrm{i}} 
 
\newcommand{\bo}{{\partial M}}

\renewcommand{\L}{\mathcal{L}}

\title[The Dirichlet-to-Neumann map on Lorentzian manifolds]{The inverse problem for the Dirichlet-to-Neumann map\\ on Lorentzian manifolds}

\author[P. Stefanov]{Plamen Stefanov}
\address{Department of Mathematics, Purdue University, West Lafayette, IN 47907}
\thanks{First author partly supported by  NSF  Grant DMS--1301646 and DMS--1600327}
\author[Yang Yang]{Yang Yang}
\address{Department of Mathematics, Purdue University, West Lafayette, IN 47907}
\begin{document}

\begin{abstract}
We  consider the Dirichlet-to-Neumann map $\Lambda$ on a cylinder-like Lorentzian manifold related to the wave equation related to the metric $g$, a magnetic field $A$ and a potential $q$. We show that we can recover the jet of $g,A,q$ on the boundary from $\Lambda$ up to a gauge transformation in a stable way. We also show that $\Lambda$ recovers the following three invariants in a stable way: the lens relation of $g$, and the light ray transforms of $A$ and $q$. Moreover, $\Lambda$ is an FIO away from the diagonal with a  canonical relation  given by  the lens relation.  We present applications for recovery of $A$ and $q$ in a logarithmically stable way in the Minkowski case, and uniqueness with partial data. 
\end{abstract}

\maketitle

\section{Introduction and main results}\label{sec_Intro}
Let $(M,g)$ be a Lorentzian manifold of dimension $1+n$, $n\ge2$, i.e., $g$ is a metric with signature   $(-1,1,\dots,1)$. Suppose a part of  $\partial M$ is timelike. An example of $M$ is a cylinder-like domain representing a moving and shape changing compact manifold in the $x$-space (if we have fixed time and space variables) with the requirement that the normal speed of the boundary is less than one, see section~\ref{sec_ex}.

% an expanding cylinder (with expanding speed less than 1) in the Minkowski space . Write $x=(x^0,\dots,x^n)$, $x'=(x^0,\dots,x^{n-1})$. Einstein's summation notation will be systematically used, with the agreement that the Roman letters ``$j,k,l,\dots$'' range over $\{0, 1,\dots, n\}$, and the Greek indice ``$\alpha,\beta,\gamma,\dots$'' ranges over $\{0, 1,\dots, n-1\}$.

Denote the wave operator by $\Box_g$; in local coordinates $x=(x^0,\dots,x^n)$ it takes the form:
$$
\Box_g := \frac{1}{\sqrt{|\det g|}} \partial_j \left( \sqrt{|\det g|} g^{jk} \partial_k \right).
$$
Consider the following operator $P=P_{g,A,q}$ which is a first order perturbation of $\Box_g$:
\begin{equation} \label{P}
P=P_{g,A,q}:= \frac{1}{\sqrt{|\det g|}}\left(\partial_j - i A_j\right) \sqrt{|\det g|} g^{jk} \left( \partial_k -i A_k \right)  + q.
\end{equation}
Here $i=\sqrt{-1}$; $A$ is a smooth $1$-form on $M$; $q$ is a smooth function on $M$. 
%\HOX{\Q: Do we need $A$ and $q$ to be real-valued so that $P$ is self-adjoint on $L^2(M)$?}

The goal of this work is to study  the inverse problem of recovery of $g$, $A$ and $q$, up to a data preserving gauge transformation,  from the outgoing  Dirichlet-to-Neumann (DN) $\Lambda$ map on a timelike boundary associated with the wave equation
\be{we}
Pu=0\quad\text{in $M$}.
\ee
We are motivated by applications in relativity but also in applications to classical wave propagation problems with media moving and/or changing at a speed not negligible compared to the wave speed.   
We are  interested in possible stability results even though some steps in the recovery are inherently unstable. This problem remains widely open. The results we prove are the following. First, we show that one can recover the jet of $g,A,q$ at  the boundary (up to a gauge transform) in a H\"older stable way. Next, we show that  one can extract the natural geometric invariants of $g,A,q$ from $\Lambda$ in a H\"older stable way. More precisely, $\Lambda$ recovers the lens relation $\mathcal{L}$  related to $g$, in stable way. If we know $g$,   the light ray transform $L_1A$ of $A$ is recovered stably. If $g$ and $A$ are known, the light ray transform  $L_0q$ of  $q$ is recovered stably. The lens relation $\mathcal{L}$ is the canonical relation of the Fourier Integral Operator (FIO) $\mathcal{L}$ away from the diagonal, and the light ray transforms $L_1A$ and $L_0q$ are  in fact encoded in the principal and the subprincipal symbol of it.
  %We do not prove a stable recovery of $\mathcal{L}$ from $\Lambda$ because 
 In fact, $\mathcal{L}$ is directly measurable from $\Lambda$.

Since the results we prove are local or semilocal (near a fixed lightlike geodesic); and the proofs are microlocal, we  do not formulate a global mixed problem for the wave equation at the beginning but we do consider one in section~\ref{sec_ex}. In fact,  existence of solutions of such problems depend on global properties of $(M,g)$, one of them is global hyperbolicity,  which are not needed for our weaker formulation and for the proofs. Instead, we  define the DN map up to smoothing operators only. In case when one can prove the existence of a global solution, the true DN map would coincide with ours up to a smoothing error, see section~\ref{sec_ex}; and our results are not affected by adding smoothing operators. 

This problem has a long history in  the stationary Riemannian setting, i.e., when $M=[0,T]\times M_0$, where $(M_0,g)$ is a compact Riemannian manifold with boundary, and the metric is $-\d t^2+g_{ij}(x) \d x^i \d x^j$. The boundary control method \cite{Belishev_87} and the Tataru's uniqueness continuation theorem \cite{tataru95,Tataru99} provide uniqueness provided that $T$ is greater than a certain sharp critical value $T$, as shown by Belishev and Kurylev in \cite{BelishevK92}, see also the survey \cite{belishev_2007}.  Stability however does not follow from such arguments. Stability results for recovering of the metric and lower order terms appeared in \cite{SU-IMRN,SU-JFA, Carlos_12,BellassouedDSF,BaoZhang}, with \cite{Carlos_12} covering the general case. A main assumption in those works is that the metric is simple, i.e., that there are no conjugate points and the boundary is strictly convex (not so essential assumption) and the main technical tool for recovery of the metric is to reduce it to stability for  the boundary/lens rigidity problem, see, e.g., \cite{SU-JAMS}. 
For related results, we refer to \cite{IsakovSun92, Sun_90}. Recently, the progress in treating the local rigidity problem allowed results under the more general foliation condition \cite{SUV-DNmap2014} which allows  conjugate points. In any case, some condition is believed to be necessary for stability. It is worth noticing that all inverse (hyperbolic) \textit{scattering} problems for compactly supported perturbations are equivalent to inverse DN map problems.

Recently, there has been  increased interest in this problem or in related inverse scattering problems in time-space. Recovery of lower order time-dependent terms for the Minkowski metric has been studied in  \cite{MR1004174,Ramm-Sj, Ramm_Rakesh_91,waters2014stable, Salazar_13, Aicha_15, Bellassoued_BA_16}, and for $-\d t^2+g_{ij}(x)\d x^i \d x^j$ in \cite{Kian_Oksanen_16}. In \cite{Eskin2015}, Eskin proved that one can recover $g,A,q$ up to a gauge transformation, assuming existence of a global time variable $t$ and  analyticity of all coefficients with respect to it. The proof is based in an adaptation of the boundary control methods and the analyticity is needed so that one can still use the unique continuation results in \cite{Tataru99}. Stability does not follow from such arguments. 
%  In \cite{LOSU-strings}, a linearized problem of recovery a Lorentzian metric is studied in view of applications to cosmology.  
 Other inverse problems on Lorentzian manifolds are studied in \cite{KLU-Einstein-I,KLU-2015,  LassasUW_2016}. The inverse scattering problem of recovery a moving boundary is studied in \cite{CooperS84, MR1115177, EskinR10}. 
 The first author showed in \cite{MR1004174} that in the case of $g$ Minkowski and $A=0$, the problem of recovery of $q$ reduces to the inversion of the X-ray transform in time-space over light rays, which was shown there to be injective for functions tempered in time and uniformly compactly supported in space.   In \cite{LOSU-strings}, it is shown that the linearized metric problem leads to the inversion of a light ray transform of tensor fields. Such light ray transforms are inherently unstable however because they are smoothing on the time-like cone. They  require specialized tools for analyzing the singularities near the lightlike cone, not fully developed in the geodesic case, see  \cite{Greenleaf-Uhlmann, Greenleaf_Uhlmann90, Greenleaf_UhlmannCM}. The light ray transform has been also studied in \cite{BQ, Begmatov01, S-support2014, Kian2016}.

We describe the main results below. 
Let $x_0\in\bo$ and assume that $\bo$ is timelike near $x_0$. Then $\bo$ with the induced metric is a Lorentzian manifold as well and we choose (locally) one of the two time orientations that we call future pointing.

Let $f\in \mathcal{E}'(\bo)$ be supported near $x_0$ with $\WF(f)$ close to a fixed timelike $(x_0,\xi^0{}')\in T^*\bo\setminus 0$.
%For  $f\in \mathcal{E}'(\bo)$ supported near $x_0$ with $\WF(f)$ in the timelike cone, we 
We define the \textit{local outgoing} solution operator $f\mapsto u$, defined up to a smoothing operator, as the operator mapping $f$ to the \textit{outgoing} solution $u$ of 
\be{ibvp}
Pu\in C^\infty \quad\text{in $M$ near $x_0$}, \qquad u|_{\bo}=f\quad \text{mod $C^\infty$}.
\ee
The term ``outgoing'' here refers to the following. We chose that microlocal solution (parametrix)
 for which the singularities of the solution are required to propagate along future pointing bicharacteristics. 
We refer to section~\ref{sec_not} for more details.  On the other hand, it is ``local'' because it solves \r{ibvp} near $x_0$ only and this keeps the singularities close enough to $\bo$ without allowing them to hit $\bo$ again.

Define the associated \textit{local outgoing} Dirichlet-to-Neumann map as
\be{DN}
\Lambda_{g,A,q}^{\rm loc} f = \left(\partial_\nu u-i\langle A,\nu \rangle u \right)|_{\partial M}, %\quad \text{mod $C^\infty$},
\ee
where $\nu$ denotes the unit outer normal vector field to $\partial M$, and the equality is modulo smoothing operators applied to $f$.  %We are interested in the stability of the determination (or the determination of some integral transforms) of the Lorentzian metric $g$, the $1$-form $A$ and the smooth function $q$, from the Dirichlet-to-Neumann map $\Lambda_{g,A,q}$. 
By definition, the $\Lambda_{g,A,q}^{\rm loc}$ is defined near $x_0$ only, and in fact, in some conic neighborhood of the timelike $(x_0,\xi^0{}')$. Since the latter is arbitrary, $\Lambda_{g,A,q}^{\rm loc}$ extends naturally to the whole timelike cone on $\bo$ but we keep it microlocalized near $(x_0,\xi^0{}')$ to emphasize what we can recover given microlocal data only.

 As we  show in Theorem~\ref{thm1a}, $\Lambda_{g,A,q}^{\rm loc}$ is actually a \PDO\ on the timelike cone bundle near $x_0$. The main result about $\Lambda_{g,A,q}^{\rm loc}$ is Theorem~\ref{bdrystability}: a stability  estimate about the recovery of the boundary jets of the coefficients. 

Let $f\in \mathcal{E}'(\bo)$ have  $\WF(f)$ as above. Let $u$ as before be the parametrix in a neighborhood of the future pointing null  bicharacteristic issued from the unique future pointing lightlike covector $(x_0,\xi^0)\in T^*M \setminus 0$ with orthogonal projection $(x_0,\xi^0{}')$. Note that the direction of  $(x_0,\xi^0)$ and that of the bicharacteristic might be the same or opposite.  
Assume that this bicharacteristics hits $\bo$ again, transversely, at point $y_0$ in the codirection $\eta^0$ and let $ \eta^0{}'$ be the corresponding orthogonal tangential projection on $T_{y_0}^*\bo$. Then $(y_0,\eta^0{}')$ is timelike, as well. Let $\mathcal{U} $ and $\mathcal{V} $ be two small conic  timelike neighborhoods in $T^*\bo\setminus 0$ of $(x_0,\xi^0{}')$ and $(y_0, \eta^0{}')$, respectively. If $\mathcal{U} $ is small enough, for every timelike $(x,\xi')\in \mathcal{U} $ close to $(x_0,\xi^0{}')$, we can define $(y,\eta')$ in the same way. 
This defines the \textit{lens relation} 
\be{L}
\L : \mathcal{U} \longrightarrow \mathcal{V}, \quad \L(x,\xi')= (y,\eta'),
\ee
see Figure~\ref{DN_Lorentz_fig1}. 
By definition,  $L$ is an even map in the second variable, i.e., $L(x,-\xi') = (y,-\eta')$. If $(x,\xi')$ is future pointing (i.e., if the associated vector by the metric is such), then $(x,-\xi')$ is past-pointing but we can interpret $(y,-\eta')$ as the end point of the null geodesic with initial point projecting to $(y,-\eta')$ but moving ``backward'' w.r.t.\ the parameter over it. This property  correlates well with Theorem~\ref{thm1b} since the wave equation has two wave ``speeds'' of opposite signs. 

The map $\L$ is positively homogeneous of order one in its second variable. Now, for $f$ as above, let $u$ be the outgoing solution to \r{ibvp} near the bicharacteristic issued from $(x_0,\xi^0)$ all the way to its second contact with $\bo$ at $y_0$. At this point, we assume that $(x_0,\xi^0{}')$ is not a fixed point for $\L$, which means that the reflected bicharacteristic does not become a periodic one after the first reflection. 
Since $f$ is smooth near $(y_0,\eta^0{}')$ that means no singularity of the  solution $u$ at $(y_0, \eta^0{}')$, therefore, the singularity reflects at $y_0$. We extend the solution microlocally over a small segment of the reflected ray before reaching $\bo$ again. Then we define the \textit{global DN map}  $\Lambda_{g,A,q}^{\rm gl}$  by \r{DN} again but with the r.h.s.\ localized  to $V$, the projection of $\mathcal{V}$ to the base.  In fact, by propagation of singularities,  $\Lambda_{g,A,q}^{\rm gl}f$ has a wave front set in $\mathcal{V}$ only and we can cut smoothly outside some neighborhood of $y_0$.   The map $\Lambda_{g,A,q}^{\rm gl}$ is actually just semi-global because it is the DN map restricted to a solution near one geodesic segment connecting boundary points. In Theorem~\ref{thm1b}, we prove that $\Lambda_{g,A,q}^{\rm gl}$ is an FIO associated with the graph of $\mathcal{L}$. In Theorem~\ref{thm_LR_stab}, we show that $\Lambda_{g,A,q}^{\rm gl}$ recovers $\mathcal{L}$ in a stable way, which is also a general property of FIOs associated to a local canonical diffeomorphism.

Another fundamental object is the \textit{light ray transform} $L$ which integrates functions or more generally tensor fields along lightlike geodesics. We define $L$ on functions by
\be{L1} 
L_0f(\gamma) =\int f(\gamma(s))\,\d s,
\ee
and on covector fields of order one by
\be{L2} 
L_1f(\gamma) =\int \langle f(\gamma(s)),\dot\gamma(s)\rangle \,\d s,
\ee
where  $\langle f(\gamma(s)),\dot\gamma(s)\rangle= f_j(\gamma(s))\dot\gamma^j(s)$ in local coordinates and $\gamma$ runs over a give set of lighlike geodesics, and we always assume that $\supp f$ is such that the integral is taken over a finite interval. In out results below, $\gamma$'s in $L_0$ and $L_1$ are the maximal geodesics through $M$ connecting boundary points.  Unlike the Riemannian case, lightlike geodesics do not have a natural speed one parameterization and every rescaling of the parameter along them (even if that rescaling changes from geodesic to geodesic) keeps them being lightlike. 
The transform $L_1$ is invariant under reparameterization of the geodesics and can be considered as an integral of $\langle f,\d\gamma\rangle$ over the geodesics. On the other hand, $L_0$ is not. Despite that freedom, the property $L_0f=0$  does not change. One way to parameterize it is to define it locally near a lightlike geodesic hitting a timelike surface at $s=0$, in our case, $\bo$. Then the orthogonal projection $\dot\gamma'(0)$ of each such $\gamma$ on $T\bo$ (the prime stands for projection) determines $\dot\gamma(0)$ and therefore, $\gamma$ uniquely. 
To  normalize the projections on $T\bo$, we can choose a timelike covector field $Z$ on $T\bo$ locally and  require  $g(\dot\gamma, Z)=\mp1$ for future/past pointing directions. 

In Theorem~\ref{intstability}, we show that given $g$, one can recover $L_1A$ in a H\"older stable way; and if we are given $g$, $A$, one can recover $L_0q$ in  a H\"older stable way.  Notice that we do not require absence of conjugate points and we do not use Gaussian beams. Instead, we use standard microlocal tools including Egorov's theorem.  In section~\ref{sec_ex}, we consider some cases where $L_1$ and $L_0$ can be inverted to derive uniqueness results. As we mentioned above, those transforms are unstable. The reason is that  they are microlocally smoothing in the spacelike cone, see, e.g., \cite{Greenleaf_UhlmannCM, S-support2014, LOSU-strings}. Therefore, stable recovery of $L_1A$ and $L_0q$ does not imply H\"older stable recovery of $A_1$ (up to a gauge transform) and $q$ but allows for weaker logarithmic estimates using the estimate for recovery of $q$ from $L_0q$ in the Minkowski case proven in \cite{Begmatov01}, for example. We discuss some of those possible corollaries in section~\ref{sec_ex}. Recovery of $g$ from $\mathcal{L}$ is an open problem with some results about the linearized problems obtained recently in \cite{LOSU-strings}. 

\textbf{Acknowledgments.} We would like to thank Matti Lassas for providing  some of the  references and for the useful discussions. 

\section{Preliminaries}
\subsection{Notation and terminology} \label{sec_not}
In what follows, we denote by $U$ and $V$ the projections of $\mathcal{U}$ and $\mathcal{V}$ onto the base $\bo$. We freely assume that $\mathcal{U}$ and $\mathcal{V}$, and therefore, $U$ and $V$ are small enough to satisfy the needed requirements below. 

If $\xi$ is a covector based at a point $x$ on $\bo$, we denote by $\xi'$ its orthogonal projection to $T_x^*\bo$. We routinely denote covectors on $T_x^*\bo$ by placing primes, like $\xi'$, etc., even if a priori such covector is not a projection of a given one. 

\textit{Timelike/spacelike/lightlike} vectors $v$ are the ones satisfying $g(v,v)<0$, or $g(v,v)>0$, or $g(v,v)=0$, respectively. We identify vectors and covectors by the metric. 
We choose an orientation in $U$ that we call future pointing (FP). More precisely, we choose some smooth timelike vector  $Z$ in $U$ (identified with an open set in the tangent bundle) and we call \textit{future pointing} those timelike vectors $v$ for which $g(v,Z)>0$. If we have a time variable $t$, for example, such a choice could be $Z=\partial/\partial t$. In semigeodesic coordinates $(t,x)$, see the remark after Lemma~\ref{seminormal}, FP $v=(v_0,v')$ means $v_0>0$. Notice that for the associated covector $(\tau,\xi) = gv$, we have $\tau<0$.

\begin{figure}[h!] % float placement: (h)ere, page (t)op, page (b)ottom, other (p)age
  \centering
  \includegraphics[scale=1,page=1]{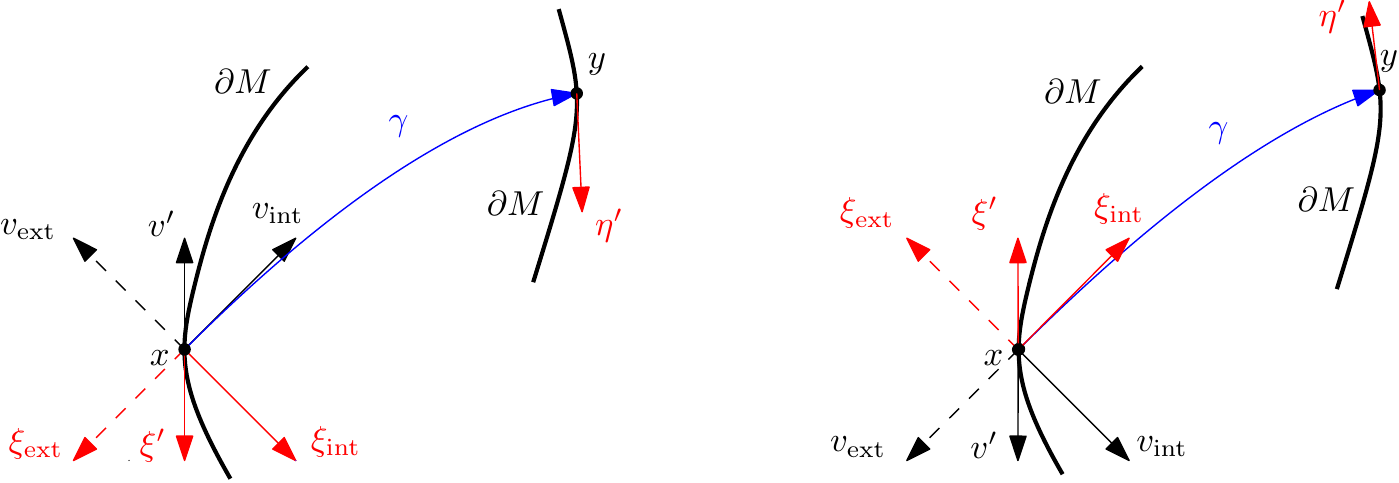}
  %trim={<left> <lower> <right> <upper>}
\caption{\small  A tangent timelike future pointing (FP) vector $v'$ on the left, and a past pointing on the right; and the two  lightlike vectors $v_{\rm int}$ and $v_{\rm ext}$ with the same projection, pointing to $M$ and outside $M$, respectively. The FP geodesic $\gamma=\gamma_{x,\xi'}(s)$ in both cases propagates to the future but on the right, it is determined by negative values of the parameter over it. The corresponding covectors $\xi'$, $\xi_{\rm int}$ and $\xi_{\rm ext}$ are plotted, as well. The lens relation is $\mathcal{L}(x,\xi')= (y,\eta')$. 
}
\label{DN_Lorentz_fig1}
\end{figure}

Given a timelike $(x,\xi')\in \mathcal{U}$, assume first that $\xi'$ is FP. Let $\xi$ be the lightlike covector pointing into $M$ with orthogonal projection $\xi'$, identified with the vector $v=g^{-1}\xi$.  The  geodesic $\gamma_{x,\xi'}(s)$ issued from $(x,v)$, for $s\ge0$ will be called  the FP  geodesic issued from $(x,\xi')$. In Figure~\ref{DN_Lorentz_fig1} on the left, $v=v_{\rm int}$ and $\gamma_{x,\xi'}(s)=\gamma$. If $(x,\xi')$ is past pointing, then we choose $v$ to be the lightlike vector projecting to $v'$ pointing to the exterior ($v_{\rm ext}$ in Figure~\ref{DN_Lorentz_fig1} on the right) and take  $\gamma_{x,\xi'}(s)$ for $s\le0$. By propagation of singularities, a boundary singularity $(x,\xi')$ as above would propagate either along the FP geodesics chosen above, or along the past pointing ones (or both) that we did not choose. The choice we made reflects the requirement that singularities should propagate to the future only. We call such microlocal solutions  \textit{outgoing}.  We borrow that term from scattering theory.  In the case of the classical formulation of the Riemannian version of this problem, this is guaranteed by the condition $u=0$ for $t<0$. 

\subsection{Gauge Invariance} 
There exist some gauge transformations which leave the local and the global versions of the Dirichlet-to-Neumann map $\Lambda_{g,A,q}$ invariant, thus one can only expect to recover the corresponding gauge equivalence class. To simplify the formulations, we assume that the DN map $\Lambda_{g,A,q}$ is well defined globally on $M$. In our main theorems, we will apply this to the \PDO\ part of  $\Lambda_{g,A,q}$  first, and then $\Phi$ below needs to be identity near a fixed point only. For the semiglobal one, we need $\Phi$ to be identity near both ends of the fixed lightlike geodesic only. Since the computations below are purely algebraic, the lemmas remain true for the localized maps  with obvious modifications. 

We will consider two types of gauge transformations in this part. The first one is a diffeomorphism in $M$ which fixes $\partial M$.

\begin{lemma} \label{changeofcoordinates}
Let $(M,g)$ be a Lorentzian manifold with boundary as above, let $A$ be a smooth $1$-form and $q$ be a smooth function on $M$. If $\Phi:M\rightarrow M$ is a diffeomorphism with $\Phi|_{\partial M}=\text{\rm Id}$, then
$$
\Lambda_{g,A,q} = \Lambda_{\Phi^{\ast}g, \Phi^{\ast}A, \Phi^{\ast}q}.
$$
Here $Id:\partial M\rightarrow\partial M$ is the identity map, $\Phi^{\ast}g, \Phi^{\ast}A, \Phi^{\ast}q$ are the pullbacks of $g,A,q$ under $\Phi$, respectively.
\end{lemma}

\begin{proof}
For any $f\in C^{\infty}(\partial M)$, let $u$ be the solution of $\mathcal{L}_{g,A,q}u=0$ on $M$ with $u|_{\partial M}=f$. Define $v:=\Phi^{\ast} u$ as the pull-back of $u$, then simple calculation in local coordinates shows that $\mathcal{L}_{\Phi^{\ast}g,\Phi^{\ast}A,\Phi^{\ast}q}v=0$ and $v|_{\partial M}=f$. If we write $y=\Phi(x)$ as a local coordinate representation of $\Phi$, then
\begin{align*}
\Lambda_{g,A,q} f(y) & = \nu^{j}(y)\frac{\partial u}{\partial y^j}(y) - i\nu^{j}(y) A_{j}(y) u(y) \Big|_{\partial M} \\
& =   \nu^j(x) \frac{\partial x^l}{\partial y^j} \frac{\partial v}{\partial y^l} - i\frac{\partial x^l}{\partial y^j} \nu^j(x) \frac{\partial y^k}{\partial x^l} A_k(x) v(x) \Big|_{\partial M}\\
& = \tilde \nu^{j}(x)\frac{\partial v}{\partial x^j}(x) - i\tilde \nu^{j}(x) (\Phi^* A)_{j}(x) v(x) \Big|_{\partial M} \\
& = \Lambda_{\Phi^{\ast}g, \Phi^{\ast}A, \Phi^{\ast}q}f,
\end{align*}
where $\nu$ and $\tilde \nu$ are the unit normals in the $y$ and the $x$ variables, respectively. The above calculation essentially verifies that $\Lambda_{g,A,q}$ is defined invariantly. 
 Therefore,  $\Lambda_{g,A,q} = \Lambda_{\Phi^{\ast}g, \Phi^{\ast}A, \Phi^{\ast}q}$.
\end{proof}

Another type of gauge invariance occurs when one makes a conformal change of the metric $g$. This type of gauge invariance also occurs when $g$ is a Riemannian metric and $\Lambda_{g,A,q}$ is the corresponding Dirichlet-to-Neumann map for the magnetic Schr\"{o}dinger equation, see \cite[Proposition 8.2]{DosSantosKSU09}.

\begin{lemma} \label{conformalgauge}
Let $(M,g)$ be a Lorentzian manifold with boundary as above, let $A$ be a smooth $1$-form and $q$ be a smooth function on $M$. If $\varphi$ and $\psi$ are smooth functions such that
$$
\varphi|_{\partial M} = \partial_\nu \varphi |_{\partial M} = 0, \quad \psi|_{\partial M}=0,
$$
then we have
$$
\Lambda_{g,A,q} = \Lambda_{e^{-2\varphi}g, A-d\psi, e^{2\varphi}(q-q_\varphi)}
$$
where $q_\varphi:=e^{\frac{n-2}{2}\varphi}\Box_g e^{\frac{2-n}{2}\varphi}$.
\end{lemma}

\begin{proof}
A direct computation in local coordinates shows that %\HOX{Why $L$ script?}
\begin{align*}
e^{\frac{n+2}{2}\varphi} P_{g,A,q} (e^{\frac{2-n}{2}\varphi}u) &=  P_{e^{-2\varphi}g, A, e^{2\varphi}(q-q_\varphi)}u \\
e^{-i\psi} P_{g,A,q} (e^{i\psi}u) & =  P_{g, A-d\psi, q}u 
\end{align*}
For any $f\in C^{\infty}(\partial M)$, let $u$ be the solution of $P_{g,A,q}u=0$ on $M$ with $u|_{\partial M}=f$. Setting $v:=e^{\frac{n-2}{2}\varphi} e^{-i\psi} u$,  we have
\begin{align*}
P_{e^{-2\varphi}g, A-d\psi, e^{2\varphi}(q-q_\varphi)}v &=  P_{e^{-2\varphi}g, A-d\psi, e^{2\varphi}(q-q_\varphi)}(e^{\frac{n-2}{2}\varphi} e^{-i\psi} u) \\
&=  e^{\frac{n+2}{2}\varphi} P_{g,A-d\psi,q}(e^{-i\psi}u) \\
&=  e^{\frac{n+2}{2}\varphi} e^{-i\psi} P_{g,A,q}u =0
\end{align*}
Furthermore, notice that $\nu_{e^{-2\varphi}g}=\nu_g$ by the assumption on $\varphi$, thus
\begin{align*}
\Lambda_{e^{-2\varphi}g, A-d\psi, e^{2\varphi}(q-q_\varphi)} f & =  \nu^j \frac{\partial v}{\partial x^j} - i \nu^j \left(A_j - \frac{\partial\psi}{\partial x^j}\right)v |_{\partial M} \\
&=  \nu^j \frac{\partial (e^{\frac{n-2}{2}\varphi} e^{-i\psi} u)}{\partial x^j} - i \nu^j \left(A_j - \frac{\partial\psi}{\partial x^j}\right) (e^{\frac{n-2}{2}\varphi} e^{-i\psi} u) |_{\partial M} \\
&=  \nu^j \left( -i\frac{\partial\psi}{\partial x^j} u + \frac{\partial u}{\partial x^j} \right) - i \nu^j i \nu^j \left(A_j - \frac{\partial\psi}{\partial x^j}\right) u |_{\partial M} \\
&=  \nu^j \frac{\partial u}{\partial x^j} - i \nu^j A_j u |_{\partial M} \\
&=  \Lambda_{g,A,q}f
\end{align*}
which completes the proof.
\end{proof}

\subsection{Gauge equivalent modifications of $g,A,q$}

It is convenient to work in semi-geodesic normal coordinates on a Lorentzian manifold. These coordinates are the Lorentzian counterparts of the well known Riemannian semigeodesic  coordinates for Riemannian manifolds with boundary. We formulate the existence of such coordinate in the following lemma.
 
\begin{lemma} \label{seminormal}
Let $S$ be either a timelike  hypersurface in $M$. For every $x_0\in S$, there exist $\eps>0$, a neighborhood $N$ of $x_0$  in $M$, and a diffeomorphism $\Psi:S\cap N \times [0,T)\rightarrow N$ such that

 (i)  $\Psi(x',0)=x'$ for all $x'\in S\cap N$;
 
(ii)  $\Psi(x',x^n) = \gamma_{x'}(x^n)$ where $ \gamma_{x'}(x^n)$ is the unit speed geodesic issued from $x'$ normal to $S$.	

Moreover,   if $(x^0,\dots,x^{n-1})$ are local boundary coordinates on $S$, in the coordinate system \allowbreak $(x^0,\dots,x^n)$,  the metric tensor $g$ takes the form
\begin{equation} \label{metricform}
g=g_{\alpha\beta} dx^\alpha\otimes dx^\beta + dx^n\otimes dx^n, \quad \alpha,\beta\le n-1. 
\end{equation}
\end{lemma}
 
Clearly, $g_{\alpha\beta}$ has a Lorentzian signature as well.  If $M$ has a boundary, then $S$ can be $\bo$ and $x^n$ is restricted to $[0,\eps]$.  
A proof of the lemma can be found in \cite{Petrov_book} and is based on the fact that the lines $x'=\text{const.}$, $x^n=s$ are unit speed geodesics; therefore the Christoffel symbols  $\Gamma_{nn}^i$ vanish for all $ i$.  
 We will call such coordinates  the semi-geodesic normal coordinates. The lemma remains true if $S$ is spacelike with a negative sign in front of $dx^n\otimes dx^n$ in \r{metricform}, and this gives as a way to define a time function $t=x^n$ locally, and put the metric in the block form \r{metricform}.

Now we use the gauge invariance of $\Lambda_{g,A,q}$ to alter $g,A,q$ without changing the DN map. Three types of modifications are made in the following, labeled as \textbf{(M1)-(M3)} respectively.

Firstly, given two metrics $g$ and $\tilde{g}$, one can choose diffeomorphisms  as in Lemma \ref{changeofcoordinates} to obtain common semi-geodesic normal coordinates. In fact, let $\Psi$ and $\tilde{\Psi}$ be diffeomorphisms like in Lemma \ref{seminormal} with respect to $g$ and $\tilde{g}$ respectively, then $\tilde{\Psi} \circ \Psi^{-1}$ is a diffeomorphism near $\partial M$ which fixes $\partial M$.  Extend $\tilde{\Psi} \circ \Psi^{-1}$ as in \cite{Palais60} to be a global diffeomorphism on $M$. %\HOX{do we need a global extension?}
The properties of $\Psi$ and $\tilde{\Psi}$ ensure that the two metrics $g$ and $(\tilde{\Psi} \circ \Psi^{-1})^\ast \tilde{g}$ have common semi-geodesic normal coordinates near $\partial M$. Therefore, we may assume \\

\textbf{(M1)}: if $(x',x^n)$ are the semi-geodesic normal coordinates for $g$, they are also the semi-geodesic normal coordinates for $\tilde{g}$.\\

Secondly, we employ the conformal gauge invariance to replace $\tilde g$ with a gauge equivalent one to obtain some identities which later will help simplify the calculations. 

\begin{lemma} \label{detvanish}
Let $S$ be either a timelike or spacelike hyperplane near some point $p_0\in S$. Given smooth functions $r_2, r_3, \dots$ on $S$ near $p_0$, there exists a smooth function $\mu>0$ near $p_0$ with $\mu=0$, $\partial_\nu\mu=0$ on $S$ so that if $\hat{\Psi}$ is the diffeomorphism in Lemma \ref{seminormal} related to the metric $\hat{g}:=e^{\mu}g$, then 
$$
\partial^j_n \det(\hat{\Psi}^\ast \hat{g}) = r_j, \quad\quad j=2,3,\dots
$$
on $S$ near $p_0$. Here $\partial_n=\frac{\partial}{\partial x^n}$ with $(x^0,\dots, x^n)$ the semi-geodesic normal coordinates for $g$.
\end{lemma}

Before giving the proof of the lemma, we remark that $(x^0,\dots,x^n)$ may not be the semi-geodesic normal coordinates for $\hat{g}$.

\begin{proof}
The statement of the theorem is invariant under replacing $g$ by $\Psi^\ast g$ for any local diffeomorphism $\Phi$ which preserves the boundary pointwise. Therefore, we may assume that $g$ is replaced by $\Psi^\ast g$, i.e., that $x=(x',x^n)$ are semi-geodesic coordinates of $g$.

Note first that the conformal factor does not change the property of a covector being normal to $S$ but rescales the normal derivative and may change the higher order ones because $\gamma_{x'}$ may change its curvature with respect to the old metric. More precisely, for the vector $e_n=(0,\dots,0,1)$ we have $g(e_n,e_n)=\mp 1$ but $\hat{g}(e_n,e_n)=\mp e^{\mu}$. Therefore, for the corresponding normal derivatives we have $\hat{\partial}_\nu=e^{-\mu/2}\partial_\nu=\partial_n$ on $x^n=0$. Let $\hat{\gamma}_{x'}(s)$ be the normal geodesic at $x'\in S$ with $\dot{\hat{\gamma}}_{x'}$ consistent with the orientation of $S$, normalized by $\hat{g}(\dot{\hat{\gamma}}_{x'}(s), \dot{\hat{\gamma}}_{x'}(s))=\mp 1$. Then for every smooth function $f$, 
$$
\partial^j_n \hat{\Psi}^\ast f (x')|_{x^n=0} = \partial^j_n|_{x^n=0} f(\hat{\gamma}_{x'}(x^n)).
$$
For $j=0,1$, the results are not affected by the conformal factor and we get 
$$\hat{\Psi}^\ast f(x')|_{x^n=0} = f(x',0), \quad \partial_n \hat{\Psi}^\ast f(x')|_{x^n=0} = f_n(x',0)$$
To compute the higher order normal derivatives, we write
\begin{equation} \label{A1}
\partial^2_n \hat{\Psi}^\ast f(x') = f_{ij} \dot{\hat{\gamma}}^{i}_{x'} \dot{\hat{\gamma}}^{j}_{x'} + f_i \ddot{\hat{\gamma}}^{i}_{x'} \quad \text{ on } x^n=0.
\end{equation}
Under the conformal change of the metric, the Christoffel symbols are transformed by the law
$$\hat{\Gamma}^{k}_{jk} = \Gamma^{k}_{ij} + \frac{1}{2}\delta^k_i \partial_j \mu + \frac{1}{2}\delta^k_j \partial_i \mu - g_{ij} \nabla^k \mu.$$
In particular,
\begin{equation} \label{A1g}
\hat{\Gamma}^{k}_{nn} = \Gamma^{k}_{nn} + \frac{1}{2}\delta^k_n \partial_n \mu + \frac{1}{2}\delta^k_n \partial_n \mu - g_{nn} \nabla^k \mu = \delta^k_n \partial_n \mu - \frac{1}{2} g^{kl}\partial_l \mu.
\end{equation}
Therefore, $\hat{\Gamma}^{k}_{nn} = 0$ on $x^n=0$ and \eqref{A1} reduces to
\begin{equation} \label{A2}
\partial^2_n \hat{\Psi}^\ast f(x') = f_{nn} \quad \text{ on } x^n=0.
\end{equation}
In a similar way, we may compute $\partial^j_n \hat{\Psi}^\ast f(x')$ on $x^n=0$. The result is $\partial^j_n f$ plus normal derivatives of $f$ of order $j-1$ and less with coefficients depending on the normal derivatives of $\mu$ up to order $j-1$. For our purposes, the exam expression does not matter.

The metric $\hat{g}$ has the form
$$(\hat{\Psi}^\ast \hat{g})_{kl} = (\hat{g}_{ij}\circ\hat{\Psi}) \frac{\partial \hat{\Psi}^i}{\partial x^k} \frac{\partial \hat{\Psi}^j}{\partial x^l} = (\hat{g}_{\alpha\beta}\circ\hat{\Psi}) \frac{\partial \hat{\Psi}^\alpha}{\partial x^k} \frac{\partial \hat{\Psi}^\beta}{\partial x^l} + \frac{\partial \hat{\Psi}^n}{\partial x^k} \frac{\partial \hat{\Psi}^n}{\partial x^l}$$
where the Greek indices range from $0$ to $n-1$ (but not $n$). In particular,
\be{Adet}
\det\hat  \Psi^* \hat g = (\det \d\hat \Psi)^2\det( \hat  g\circ \hat \Psi).
\ee

We need to understand the structure of $\partial_n^k(\det \d\hat\Psi)|_{x^n=0}$ now. For $k=0$, we have $\d\hat\Psi|_{x^n=0}=\Id$. Notice next that
\be{A2d}
\d\hat\Psi = (\partial_0\hat\Psi,\dots,\partial_{n-1}\hat\Psi, \partial_n\hat\Psi),
\ee
where each partial derivative is a vector. Since by \r{A1g}, $\partial_n^2\hat \Psi^i =  -\hat\Gamma^i_{nn}=0$ for $x^n=0$,
\[
\partial_n(\det \d\hat\Psi)|_{x^n=0} = 0.
\]
To analyze $k=2$, we notice first that
\[
\partial_n^3\hat \Psi^i =  -\partial_n\hat\Gamma^i_{nn}= -\partial_n \left(\delta_n^i \partial_n\mu - \frac12 g^{il}\partial_l \mu\right) = -\delta_n^i \mu_{nn}+\dots,
\]
where the dots represent a term involving lower order $\partial_n$ derivatives of $\mu$. Using this in \r{A2d}, we get
\[
\partial_n^2(\det \d\hat\Psi)|_{x^n=0} = -\mu_{nn}|_{x^n=0}.
\]
Reasoning as above, we see that
\be{A2j}
\partial_n^j(\det \d\hat\Psi)|_{x^n=0} = -\partial_n^j \mu|_{x^n=0}+\dots,
\ee 
where the dots represent terms involving normal derivatives of $\mu$ (possibly differentiated tangentially) up to order $j-1$.

We will analyze the normal derivatives of $\det(\hat g\circ\hat \Psi)$  in \r{Adet} now. 
 Since $\det \hat g = e^{n\mu}\det g$, we get
\be{Amu}
\begin{split}
 \partial_n \det(\hat g\circ \hat \Psi) &= \partial_n \left( e^{(n+1)\mu\circ \hat \Psi} \det g \circ \hat \Psi\right)\\
 & =   (n+1)\mu_n\det g+  \partial_{n} \det g  =  \partial_{n} \det g 
\quad\text{on $\partial M$}.
\end{split}
\ee
We used the fact that $\d\Psi=\Id$ on $\partial M$ and that $\partial_n\d\Psi=0$ since $\d\mu=0$ on $\partial M$.  
Therefore, $ \partial_n^j \det \hat g\circ \hat \Psi=  \partial_n^j \det  g $ on $x^n=0$ for $j=0,1$.

For the highest order derivatives, notice that $\partial_n^j \hat\Psi$ involves $\partial_n^{j-1}\mu$ as its highest order normal $\mu$ derivative, as the arguments leading to \r{A2j} show. Differentiating \r{Amu}, we therefore get 
\be{Amu2}\begin{split}
 \partial_n^j \det(\hat g\circ \hat \Psi)& = \partial_n^j \left( e^{(n+1)\mu\circ \hat \Psi} \det g \circ \hat \Psi\right)\\
 & =   (n+1)(\partial_n^j \mu) \det g+  \dots 
\quad\text{on $\partial M$},
\end{split}
\ee
where the dots have the same meaning as in \r{A2j}. 

Use \r{Adet} in combination with \r{A2j} and \r{Amu2} to get
\be{A3}
\partial_n^j (\det\hat  \Psi^* \hat g)|_{x^n=0}=    (n-1)(\partial_n^j \mu) \det g+  \dots,
\ee

To complete the proof of the lemma, we determine the normal derivatives of $\mu$ on $x^n=0$ for $j=2,\dots$. We get first $\partial_n^2 (\det\hat  \Psi^* \hat g)|_{x^n=0}= (n-1)\mu_{nn}|_{x^n=0} $, which needs to be equal to $r_2$; and can be solved for $\mu_{nn}$. Then we can determine the tangential derivatives of the latter. After that, we can solve \r{Amu2} with $j=3$ for $\mu_{nnn}$, etc. To complete the proof, we use Borel's lemma. 
\end{proof}

Let $g$ and $\tilde{g}$ be two metrics satisfying \textbf{(M1)} with the two diffeomorphisms $\Psi$ and $\tilde{\Psi}$ respectively as in Lemma \ref{seminormal}. Applying Lemma \ref{detvanish} to $S=\partial M$ and $p=x_0$, we can find a metric $\hat{g}:=e^{\mu}g$ with $\mu=0$, $\partial_\nu \mu=0$ on $\partial M$ such that under the semi-geodesic normal coordinates $(x^0,\dots,x^n)$ for $g$ we have
$$
\partial^j_n \det(\hat{\Psi}^\ast \hat{g}) = \partial^j_n \det(\tilde{\Psi}^\ast \tilde{g})  \quad\quad j=2,3,\dots.
$$
on $\partial M$. Notice that $(x^0,\dots,x^n)$ are also semi-geodesic normal coordinates for $\tilde{g}$ by \textbf{(M1)}.

Now consider the metrics $(\hat{\Psi}\circ \tilde{\Psi}^{-1})^\ast \hat{g}$ and $\tilde{g}$. These metrics have common semi-geodesic normal coordinates (see the argument following Lemma \ref{seminormal}), which are $(x^0,\dots,x^n)$. In these coordinates the choice of $\hat{g}$ yields
$$\partial^j_n \det (\tilde{\Psi}^\ast \circ (\hat{\Psi} \circ \tilde{\Psi}^{-1})^\ast \hat{g} ) = \partial^j_n \det (\hat{\Psi}^\ast \hat{g}) = \partial^j_n \det (\tilde{\Psi}^\ast \tilde{g}).$$
Thus we may replace $g$ by $(\hat{\Psi}\circ \tilde{\Psi}^{-1})^\ast \hat{g}$ and change $A$, $q$ accordingly as in Lemma \ref{changeofcoordinates} and Lemma \ref{conformalgauge} without affecting $\Lambda_{g,A,q}$. We therefore can assume that $g$ and $\tilde{g}$ satisfy not only \textbf{(M1)}, but also \\
%Now consider the metrics $\hat{g}$ and $(\tilde{\Psi}\circ \hat{\Psi}^{-1})^\ast \tilde{g}$. These metrics have common semi-geodesic normal coordinates (see the argument following Lemma \ref{seminormal}), say $(x',x^n)$, and in these coordinates we have
%$$\partial^j_n \det (\hat{\Psi}^\ast \hat{g}) = \partial^j_n \det (\tilde{\Psi}^\ast \tilde{g}) = \partial^j_n \det (\hat{\Psi}^\ast \circ (\tilde{\Psi} \circ \hat{\Psi}^{-1})^\ast \tilde{g} ).$$
%Thus if $\Lambda_{g,A,q}=\Lambda_{\tilde{g},\tilde{A},\tilde{q}}$, replacing $g$ by $\hat{g}$ and $\tilde{g}$ by $(\tilde{\Psi}\circ \hat{\Psi}^{-1})^\ast \tilde{g}$ if necessary, we may assume $g$ and $\tilde{g}$ satisfy not only \textbf{(M1)}, but also \\

\textbf{(M2)}: in the common semi-geodesic normal coordinates $(x',x^n)$,  
%\vspace{-3mm}
$$\partial^j_n \det g (x',0) = \partial^j_n \det \tilde{g} (x',0)   \quad\quad j=2,3,\dots.$$
Here we have identified the metrics with their coordinate representations under $\tilde{\Psi}$. \\%These replacements of metrics do not alter the gauge equivalence class of $\Lambda_{g,A,q}$ and $\Lambda_{\tilde{g},\tilde{A},\tilde{q}}$ thanks to Lemma \ref{changeofcoordinates} and Lemma \ref{conformalgauge}.\\

Thirdly, we make modifications to the $1$-form $A$. Again the modification does not change the gauge equivalence class of $\Lambda_{g,A,q}$ due to Lemma \ref{conformalgauge}.

\begin{lemma} \label{1formvanish}
Let $(M,g)$ be a Lorentzian manifold with boundary as above, let $A$ be a smooth $1$-form and $q$ be a smooth function on $M$. There exists a smooth functions $\psi$ with $\psi|_{\partial M}=0$ such that in the semi-geodesic normal coordinates $(x',x^n)$, $B:=A-d\psi$ satisfy
\begin{equation} \label{normalvanish}
\partial^j_n B_n (x',0) = 0  \quad\quad j=0,1,2,\dots.  
%\partial^j_n (\log (-\det h_{\alpha\beta}))(x',0) = 0  & \quad\quad j=2,3,4,\dots \label{logvanish}.\\
%\partial^j_n (h_{\alpha\beta} \partial_n h^{\alpha\beta})(x',0) = 0  & \quad\quad j=1,2,3,\dots \label{productvanish}.\\
%\partial^j_n (h_{\alpha\beta} \partial_n h^{\alpha\beta}) = 0 & \quad\quad j=1,2,3,\dots \label{productvanish}.
\end{equation}
\end{lemma}

\begin{proof}
We can find a smooth function $\psi$ with 
$$\psi(x',0)=0, \quad  \partial^{j+1}_{n} \psi(x',0) = \partial^{j}_{n} A_n(x',0), \quad j=0,1,2,\dots.$$ 
Extend it in a suitable manner so that $\psi\in C^{\infty}(M)$ with $\psi|_{\partial M}=0$. Then $B=A-d\psi$ satisfies \eqref{normalvanish}.
\end{proof}

As a result we may further assume\\

\textbf{(M3)}: in the common semi-geodesic normal coordinates $(x',x^n)$ of $g$ and $\tilde{g}$,  
%\vspace{-3mm}
$$\partial^j_n A_n (x',0) = \partial^j_n \tilde{A}_n (x',0) = 0 \quad\quad j=0,1,2,\dots.$$

\section{Boundary stability}\label{sec_BS} 
We choose the semi-geodesic coordinates $(x',x^n)$ near $x_0$ so that $x_0=0$,    $\partial M$  locally is given by $x^n=0$, and the interior of $M$ is given by $x^n>0$. 
Let $\xi^0 {}'$ be a future pointing timelike covector in $T^*_{x_0}\bo$ at $x_0$.  On Figure~\ref{DN_Lorentz_fig1}, the associated vector would look like $v'$ on the left, while the covector $\xi^0 {}'$ would have the opposite time direction, like the figure on the right. 
Let $\chi(x',\xi') $ be a smooth cutoff function with small enough support in $\mathcal{U}$ that equals to $1$ in a smaller conic timelike neighborhood of $(x_0,\xi^0 {}')$. Assume also that $\chi$ is homogeneous in $\xi'$ of order $0$.

For
\begin{equation} \label{f}
f(x')=e^{i\lambda x' \cdot \xi'}\chi(x',\xi'),
\end{equation} 
and for every $N>0$, we would like to construct a geometric optics approximation of the outgoing solution $u$ near $x_0$ in $M$ of the form
\be{u_N}
u_{N}(x):=e^{i\lambda \phi(x,\xi')} \sum^{N}_{j=0} \frac{1}{\lambda !} a_j(x,\xi').
\ee

The eikonal and the transport equations below are based on the following identity
\[
e^{-i\lambda\phi}P  e^{i\lambda\phi}=    -\lambda^2  g^{jk} (\partial_j \phi)( \partial_k\phi) + i \lambda P\phi+ 2i\lambda  
 g^{jk} \partial_j \phi (\partial_k -iA_k) + P .
\]

In $M$ near $x_0$, the phase function $\phi(x,\xi')$ solves the eikonal equation, which in the semi-geodesic coordinates takes the form
\begin{equation} \label{eikonal}
g^{\alpha\beta} \partial_\alpha \phi \; \partial_\beta \phi + \left(\partial_n \phi\right)^2 =0 , \quad\quad\quad \phi|_{x^n=0}=x'\cdot\xi'.
\end{equation}
With the extra condition $\partial_{\nu}\phi |_{\partial M}<0$, %\HOX{what is $\nu$?} 
\eqref{eikonal} is locally uniquely solvable. Moreover, \eqref{eikonal} implies 
\begin{equation} \label{xinequal}
\partial_n \phi(x',0)=\xi_n(x',\xi')>0	\quad\quad\quad \text{ for any } (x',\xi')\in\mathcal{U},
\end{equation}
where
\begin{equation} \label{xin}
\xi_n(x',\xi') := \sqrt{-g^{\alpha\beta}(x) \xi_\alpha \xi_\beta }.
\end{equation}
Notice that the choice of the sign of $\xi_n$ makes $\xi$ a lightlike future-pointing covector, pointing into $M$. In Figure~\ref{DN_Lorentz_fig1}, the associated vector $v=g^{-1}\xi$ looks like $v_{\rm int}$ on the left.

We recall briefly the method of characteristics for solving the eikonal equation. We first determine $\partial\phi$ on $x^n=0$ to get \r{xinequal} or the same equation with a negative square root. We choose one of them, and in this case our choice is determined by the requirement that %the singularities propagate into the future. 
$\partial\phi$ points into $M$, see Figure~\ref{DN_Lorentz_fig1}.    Let now $(q_{x',\xi'}(s), p_{x',\xi'}(s))$ be the null bicharacteristic with $q_{x',\xi'}(0)=x'$, $p_{x',\xi'}(0)= (\xi',\xi_n)$. We think of $(x',s)$ as local coordinates and set $\phi(x',s)=x'\cdot\xi'$. More precisely, $\phi$ is uniquely determined locally by the requirement to be constant along the null bicharacteristics $q_{x',\xi'}$. 
%$\phi(x,\xi')$ is the unique local solution of $\gamma_{z',\xi'}(s)=x$ for $(s,z')$. By the implicit function theorem, this can be done locally.  
%The eikonal equation is satisfied on $x^n=0$ and one can check by a direct differentiation that the left hand side of \r{eikonal}, which is just  $2H(x,\d \phi)$, where $H(x,\xi)=\frac12 g^{\alpha\beta}\xi_\alpha\xi_\beta$ us the Hamiltonian,  remains constant along those geodesics. 
Moreover,
\be{Ham}
%x'= \nabla_\xi \phi(q(s),\xi'), \quad 
p(s) = \nabla_x \phi(q(s),\xi').
\ee
%see, e.g., \cite[Theorem~VIII.5.1]{Taylor-book0} and the discussion after it. 
Since by the Hamilton equations, $\dot q^i(s) =g^{ij}p_j(s)$,  we get in particular that $g^{ij} \partial_j\phi\partial_i$ is just the derivative $\partial/\partial s$ along the null bicharacteristic. 

In $M$ near $x_0$, the amplitudes $a_0$ and $a_j, j=1,2,\dots$ solve the following transport equations:
\begin{align}
Ta_0 = & 0,  \quad\quad a_0|_{x_n=0}=\chi; \label{transporta0}\\
iTa_j = & -P a_{j-1}, \quad\quad a_j|_{x_n=0}=0; \quad j\geq 1. \label{transportaj}
\end{align} 
where the operator $T$ is defined as
\begin{equation} \label{transportoperator}
T:=2 g^{jk} \partial_j \phi  \left( \partial_k - iA_k\right) + \Box_g\phi.
\end{equation}
We prefer to express  the bicharacteristics through the geodesics  $\Gamma(s) := (q_{x',\xi'}(s), p_{x',\xi'}(s))=  (\gamma_{x',\xi'}(s), g\dot \gamma_{x',\xi'}(s))$. Then along the bicharacteristics, we have 
\be{T}
T = 2\partial_s - 2\i \langle A,p(s)\rangle+  \Box_g\phi = 2 \mu\partial_s \mu^{-1},
\ee
with the integrating factor $\mu$  given by 
\begin{equation} \label{mu}  %{azeroinc}
\begin{split}
%a^{\rm inc}_{0}(\Gamma(s) )
\mu (\Gamma(s) ) &=  \exp\bigg\{-\frac{1}{2} \int_0^s (\Box_g\phi) (\Gamma(\sigma) )\, \d \sigma \bigg\}\\
& \qquad \times\exp\bigg\{i \int_0^s \big\langle A \circ\gamma_{x',\xi'}(\sigma), \dot \gamma_{x',\xi'}(\sigma) \big\rangle \,\d \sigma\bigg\}.
\end{split}
\end{equation} 

The amplitudes $a_j, j=0,1,\dots$ are supported in a neighborhood of the characteristics issued from $x_0\in\partial M$ in the codirection $\xi(x_0)$. 
As a result, on some neighborhood of $x_0$, $u_N$ solves $Pu_N = O(\lambda^{-N})$, $u|_{\bo}=f$.

%In the next theorem,  $\Lambda_{g,A,q}^{\rm loc}$ is extended to the whole timelike cone on 
\begin{thm}\label{thm1a}
 $\Lambda_{g,A,q}^{\rm loc}$ is an elliptic {\rm \PDO} of order $1$ in $\mathcal{U}$. 
\end{thm}

\begin{proof}
% We prove here that  $\Lambda_{g,A,q}^{\rm loc}$ is a \PDO. 

Given $f\in \mathcal{E}'(U)$ (not related to \r{f}) with a wave front set as in the theorem, we are looking for an outgoing solution $u$ of $Pu=0$ near $x_0$, $u=f$ on $U$ of the form 
\be{ansatz-h}
u (x)= (2\pi)^{-n}\int e^{i\phi(x,\xi')} a(x,\xi') \hat f(\xi')\, \d\xi'.
\ee
The phase $\phi$ solves the eikonal equation \r{eikonal} and therefore coincides with $\phi$ there. We chose the solution which guarantees an outgoing $u$, which corresponds to the positive square root in \r{xin}.  
We are looking for an amplitude $a$ of the form $a\sim \sum_{j=0}^\infty a_j(x,
\xi')  $, where $a_j$ is homogeneous in the $\xi'$ variable of degree $-j$. The standard geometric optics construction leads to the transport equations \r{transporta0}, \r{transportaj}. Using the standard Borel lemma argument, we construct a convergent series for $a$. Then $u$ is the microlocal solution (up to a microlocally smoothing operator applied to $f$) that we used to define $\Lambda_{g,A,q}^{\rm loc}$. Then $\Lambda_{g,A,q}^{\rm loc}f = \partial u/\partial \nu|_U$. Since $\phi=x'\cdot\xi'$ on $U$, we get that $\Lambda_{g,A,q}^{\rm loc}$ is a \PDO\ with symbol
\[-
i\xi_n(x',\xi')-\partial_n a|_{x^n=0}.  
\]
%on the set $\chi=1$. 
In particular,  for the principal symbol we get
\be{N0}
\sigma_p\left( \Lambda_{g,A,q}^{\rm loc}\right)(x',\xi')  = -i \xi_n = -i \sqrt{-g^{\alpha\beta}(x') \xi_\alpha \xi_\beta }. 
\ee
The proof is the same if $(x',\xi')$ is past pointing. 
\end{proof}

%\HOX{Here we need a boundary determination result, or a gauge transform, to show that one can choose common semi-geodesic coordinates to both $g$ and $\tilde{g}$.}
We prove a stable determination result on the boundary next. Let $(g,A,q)$ and $(\tilde{g}, \tilde A, \tilde q)$ be two triples. Denote 
\be{delta}
\delta=\big\|\Lambda_{g,A,q}^{\rm loc}-\Lambda_{\tilde g,\tilde A,\tilde q}^{\rm loc}\big\|_{H^1(U)\to L^2(U)},
\ee
where, as above, $\Lambda_{g,A,q}^{\rm loc}$ and $\Lambda_{\tilde g,\tilde A,\tilde q}^{\rm loc}$ are the local DN maps associated with $(g,A,q)$  and $(\tilde{g}, \tilde A, \tilde q)$, respectively microlocally restricted to a fixed conic neighborhood $\mathcal{U}$ of a timelike future pointing $(x_0,\xi^0 {}')\in T^*U$ with $x_0\in U\subset\bo$. As above, we assume that $\xi^0 {}'$ is future pointing and timelike for both $g$ and $\tilde{g}$, and that $\mathcal{U}$ is small enough so that is included in the future timelike cone on $T^*U$ for both metrics. 
%\HOX{Do we assume $g=\tilde g$ on $\bo$?}
Therefore, in the theorem below, we need to know the DN map microlocally only near a fixed timelike covector on $T^*\bo$.

\begin{thm} \label{bdrystability} 

Let $(g,A,q)$  and $(\tilde{g}, \tilde A, \tilde q)$ be replaced by their gauge equivalent triples satisfying  \textbf{(M1)-(M3)}. Then for any $\mu<1$ and $m\geq 0$, and some open neighborhood $U_0\Subset U$ of $x_0$,
\begin{enumerate}
	\item $\displaystyle\sup_{x\in \overline{U}_0, |\gamma|\leq m} |\partial^\gamma (g-\tilde{g})| \leq C \delta^{\frac{\mu}{2^m}};$
	\item $\displaystyle\sup_{x\in \overline{U}_0, |\gamma|\leq m} |\partial^\gamma (A-\tilde{A})| \leq C \delta^{\frac{\mu}{2^{m+1}}};$
	\item $\displaystyle\sup_{x\in \overline{U}_0, |\gamma|\leq m} |\partial^\gamma (q-\tilde{q})| \leq C \delta^{\frac{\mu}{2^{m+2}}};$
\end{enumerate}
are valid whenever $g,\tilde{g}, A, \tilde{A}, q, \tilde{q}$ are bounded in a certain  $C^{k}$ norm in the semi-geodesic normal coordinates near $x_0$ with a constant $C>0$ depending on that bound with $k=k(m,\mu)$.  
\end{thm}

\begin{proof}
We adapt the proofs in \cite{Carlos_12} and \cite{SU-IMRN} in the Riemannian setting. 
 Let $\Gamma_0$ be a small conic neighborhood of $\xi^0$. We can assume that $\chi=1$ on $U_0\times\Gamma_0$.  Let $f$ be as in \r{f}. We restrict $(x',\xi')$ to $U_0\times\Gamma_0$ below.  
 Since $\partial_\nu = -\partial_n$, the formal Dirichlet-to-Neumann map in the boundary normal coordinates $(x',x^n)$ is given by
\begin{equation} \label{DNmap}
\begin{split}
\Lambda_{g,A,q}^{\rm loc} f(x') =&-e^{i\lambda x'\cdot \xi'} \bigg( i\lambda \partial_n \phi (x',0,\xi') + \sum^{N}_{j=0} \frac{1}{\lambda^j} \left(\partial_n-iA_n \right)a_j (x',0,\xi) \bigg)\\
&\quad {}+O\left(\lambda^{-N-1}\right).
\end{split}
\end{equation}
%\HOX{Here we need the (missing) existence and uniqueness result to have at least one derivative in $x^n$.}
The expression for $\Lambda_{\tilde{g}, \tilde{A}, \tilde{q}}f$ is similar, with $\phi$ and $a_j$ replaced by $\tilde{\phi}$ and $\tilde{a}_j$, respectively. 

The representation \r{DNmap} could be derived from \r{u_N} but since $u$ there is an approximate solution only, and we defined $\Lambda_{g,A,q}^{\rm loc}$ microlocally, we need to go back to its definition. To justify \r{DNmap}, notice that by \cite[Ch.~VIII.7]{Taylor-book0}, on the set $\chi=1$,   $  e^{-i\lambda x'\cdot \xi'} \Lambda_{g,A,q}^{\rm loc} f$ is equal to the full symbol of $\Lambda_{g,A,q}^{\rm loc}$ with $\lambda=|\xi|$  and $\xi$ in \r{DNmap} bounded, say, unit.  

%First, we show the stable determination of the metric on the boundary. 
In the following, $C$ denotes various constants depending only on $M$, $\chi$ in \r{f}, on the choice of $k\gg 1$ and on the a priori bounds of the coefficients of $P$ in $C^k$. 
%\HOX{What is $k$?} 
Solving for $\partial_n \phi$ (resp. $\partial_n \tilde{\phi}$) in \eqref{DNmap} and taking the difference we obtain 
\begin{align*}
\partial_n \phi - \partial_n \tilde{\phi} =  & \frac{1}{i\lambda} \left( \Lambda^{\rm loc}_{g,A,q}f - \Lambda^{\rm loc}_{\tilde{g},\tilde{A},\tilde{q}}f \right) \\
 &{}+  \frac{1}{i\lambda} \sum^{N}_{j=0} \frac{1}{\lambda^j} \left[ \left(\partial_n a_j - \partial_n \tilde{a}_j \right) - i(A_n a_j - \tilde{A}_n \tilde{a}_j) \right] + O\left(\lambda^{-N-1}\right) 
\end{align*}
in $L^{2}(U_0)$. 
Integrating in $U_0$ yields 
\begin{equation} \label{estimate0}
\left\|\partial_n \phi - \partial_n \tilde{\phi} \right\|_{L^{2}(U_0)} \leq \frac{C}{\lambda} \delta \|f\|_{H^{1}(U_0)} + \frac{C}{\lambda}.
\end{equation}
The choice of $f$ in \eqref{f} indicates that $\|f\|_{H^{1}(U_0)} \leq C\lambda$. Thus, taking the limit $\lambda\rightarrow\infty$ yields
\begin{equation} \label{xinstability}
\|\xi_n-\tilde{\xi}_n\|_{L^{2}(U_0)} = \left\|\partial_n \phi - \partial_n \tilde{\phi} \right\|_{L^{2}(U_0)} \leq C\delta.
\end{equation}
From  relation \eqref{xin} we have
\begin{equation} \label{estimate1}
\left\| (g^{\alpha\beta}-\tilde{g}^{\alpha\beta}) \xi_\alpha \xi_\beta \right\|_{L^{2}(U_0)} = \|\xi^2_n - \tilde{\xi}^2_n\|_{L^{2}(U_0)} =  \left\|(\partial_n \phi )^2 - (\partial_n \tilde{\phi} )^2\right\|_{L^{2}(U_0)} \leq C\delta.
\end{equation}
It then follows from Lemma \ref{mainlemma} that by choosing appropriate timelike covectors $\xi'$, \eqref{estimate1} implies $\|g-\tilde{g}\|_{L^{2}(U_0)} \leq C\delta$. %Using a partition of unity of $\partial M$ shows $\|g-\tilde{g}\|_{L^{2}(\partial M)} \leq C\delta$. 
%\HOX{To use a partition of unity, we need $\partial M$ to be compact.}
By interpolation estimates in Sobolev space and Sobolev embedding theorems, we have for any $m\geq 0$ and $\mu<1$ that
\begin{equation} \label{bdryg}
\|g-\tilde{g}\|_{C^{m}(\overline{U}_0)} \leq C \delta^{\mu}
\end{equation}
provided $k \gg 1$ is sufficiently large.

Second, we show that the first order normal derivatives of $g$ and the $1$-form can be stably determined on the boundary. From \eqref{DNmap} we have
\begin{align*}
 & (\partial_n - i\tilde{A}_n )\tilde{a}_0 -  (\partial_n - iA_n )a_0 =  e^{-i\lambda x'\cdot\xi'} \left( \Lambda_{g,A,q}f - \Lambda_{\tilde{g},\tilde{A},\tilde{q}}f \right)+ \nonumber \vspace{1ex} \\
 &  i\lambda  (\partial_n \phi - \partial_n \tilde{\phi} )  + \sum^{N}_{j=1} \frac{1}{\lambda^j} \left( \partial_n a_j - \partial_n \tilde{a}_j \right) + O\left(\frac{1}{\lambda^{N+1}}\right) \quad \text{ in } L^{2}(U_0).
\end{align*}
Estimate as in \eqref{estimate0} to obtain
$$
\left\|( \partial_n - iA_n )a_0 - ( \partial_n - i \tilde{A}_n )\tilde{a}_0 \right\|_{L^{2}(U_0)} \leq C(\delta+\lambda\delta+\frac{1}{\lambda})
$$
which holds for all $\lambda>0$. In particular, we may choose $\lambda=\delta^{-\frac{1}{2}}$ to minimize the right-hand side, then
\begin{equation} \label{a0normal}
\left\| ( \partial_n - iA_n )a_0 - ( \partial_n - i \tilde{A}_n )\tilde{a}_0 \right\|_{L^{2}(U_0)} \leq C \delta^{\frac{1}{2}}.
\end{equation}
In order to estimate the difference of first order normal derivatives of the metrics, we consider the transport equation in \eqref{transporta0}. Since $\chi\equiv 1$ for $x\in U_0$, it follows from the boundary condition in \eqref{transporta0} that $\partial_\alpha a_0=\partial_\alpha \chi=0$ for $\alpha=0,\dots,n-1$. Moreover, $g^{nj}=\delta^{nj}$ in the semi-geodesic coordinates, thus the transport equation in \eqref{transporta0} becomes
\begin{equation} \label{transporta01}
2\xi_n \left(\partial_n - iA_n \right) a_0 -2iA^\alpha \xi_\alpha + \frac{1}{\sqrt{-\det g}} \partial_n \left( \sqrt{-\det g} \partial_n \phi \right) + Q(g) = 0,
\end{equation}
where, as before, Greek indices range from $0$ to $n-1$ (but not $n$). 
Here $A^{\alpha}:=g^{\alpha\beta}A_\beta$, and $Q(g)$ is defined as follows which is a linear combination of tangential derivatives of $g$:
$$
Q(g) := \frac{1}{\sqrt{-\det g}} \partial_\alpha \left( \sqrt{-\det g} g^{\alpha\beta} \right) \xi_\beta.
$$
%\HOX{unclear. Is that a definition of $Q$? \\ } 
where we have used that $\partial_\beta \phi=\xi_\beta$ in $U_0$, $\beta = 0,\dots,n-1$. As a consequence of \eqref{bdryg}, 
\begin{equation} \label{Testimate}
Q(g) - Q(\tilde{g}) = O(\delta^{\frac{1}{2}}).
\end{equation}
%Here and in the following we use $\mathcal{O}(\delta^{\frac{1}{2}})$ to mean ``of order $\delta^{\frac{1}{2}}$ in $L^2(U_0)$''. 
Therefore, combining \eqref{a0normal} \eqref{transporta01} and \eqref{Testimate} we obtain
$$
\frac{1}{\sqrt{-\det g}} \partial_n \left( \sqrt{-\det g} \partial_n \phi \right) -
\frac{1}{\sqrt{-\det \tilde{g}}} \partial_n \left( \sqrt{-\det \tilde{g}} \partial_n \tilde{\phi} \right)
-2 i(A^\alpha-\tilde{A}^\alpha)\xi_\alpha
= O(\delta^{\frac{1}{2}}).
$$
Notice that
\begin{align*}
\frac{1}{\sqrt{-\det g}} \partial_n \left( \sqrt{-\det g} \partial_n \phi \right) = & \frac{1}{2\det g} \partial_n \det g \partial_n \phi + \partial^2_n \phi \\
= & \frac{\xi_n}{2\det g} \partial_n \det g - \frac{1}{2\xi_n} \partial_n g^{\alpha\beta} \xi_\alpha \xi_\beta
\end{align*}
is an even function of $\xi'$. Here in the computation $\partial_n \phi$ is substituted by $\xi_n$ due to \eqref{xinequal} and $\partial^2_n \phi$ is calculated by differentiating the eikonal equation \eqref{eikonal}. Separating the even and odd parts in $\xi'$ we conclude
\begin{equation} \label{gderivative}
\left( \frac{\xi_n}{2\det g} \partial_n \det g - \frac{1}{2\xi_n} \partial_n g^{\alpha\beta} \xi_\alpha \xi_\beta \right) - \left( \frac{\tilde{\xi}_n}{2\det \tilde{g}} \partial_n \det \tilde{g}  - \frac{1}{2\tilde{\xi}_n} \partial_n \tilde{g}^{\alpha\beta} \xi_\alpha \xi_\beta \right) = O(\delta^{\frac{1}{2}});
\end{equation}
\begin{equation} \label{1form}
(A^\alpha-\tilde{A}^\alpha)\xi_\alpha = O(\delta^{\frac{1}{2}}).
\end{equation}
For the odd part \eqref{1form}, applying Lemma \ref{mainlemma} yields
\begin{equation} \label{Abdry}
\|A-\tilde{A}\|_{L^{2}(U_0)} \leq C \delta^{\frac{1}{2}}.
\end{equation}
To deal with the even part, notice \eqref{gderivative} states that 
$$\frac{\xi_n}{2\det g} \partial_n \det g - \frac{1}{2\xi_n} \partial_n g^{\alpha\beta} \xi_\alpha \xi_\beta$$
is stably determined of order $O(\delta^{\frac{1}{2}})$. As $\xi_n$ is stably determined on $V$, see \eqref{xinstability}, their product
\begin{align*}
 \frac{\xi^{2}_{n}}{2\det g} \partial_n \det g &- \frac{1}{2} \partial_n g^{\alpha\beta} \xi_\alpha \xi_\beta \\
= & -\frac{1}{2\det g} (\partial_n \det g) g^{\alpha\beta}\xi_\alpha \xi_\beta - \frac{1}{2} \partial_n g^{\alpha\beta} \xi_\alpha \xi_\beta \\
= & -\frac{1}{2} \frac{1}{\det g} \partial_n \left( \det g \cdot g^{\alpha\beta}\right) \xi_\alpha \xi_\beta
\end{align*}
is also stably determined. Since $\det g$ is known to be stable and away from zero, it follows that $\partial_n h^{\alpha\beta}$ is stable where $h^{\alpha\beta}:=(\det g)g^{\alpha\beta}$. Hence, the normal derivative of $g=(\det h)^{\frac{1}{1-n}}h$ is also stably determined, that is,

\begin{equation} \label{gnormal}
\left\| \partial_n g - \partial_n \tilde{g}\right\|_{L^{2}(U_0)} \leq C\delta^{\frac{1}{2}}.
\end{equation}
%For \eqref{1form}, applying Lemma \ref{mainlemma} yields
%\begin{equation}
%\|A-\tilde{A}\|_{L^{2}(U_0)} \leq C \delta^{\frac{1}{2}}.
%\end{equation}
Using interpolation and Sobolev embedding theorems, we obtain from \eqref{gnormal} and \eqref{Abdry} that for any $m\geq 0$ and $\mu<1$,
\begin{equation} \label{gA}
\left\| \partial_n g - \partial_n \tilde{g} \right\|_{C^{m}(\overline{U}_0)} + \|A-\tilde{A}\|_{C^{m}(\overline{U}_0)} \leq C\delta^{\frac{\mu}{2}}
\end{equation}
provided $k\gg 1$ is sufficiently large.

Next we show that the second order normal derivatives of $g$, the first order normal derivatives of $A$, and the values of $q$ can be stably determined on the boundary. By \eqref{DNmap} up to $\lambda^{-1}$ we obtain
$$
\left\| ( \partial_n - iA_n )a_1 - ( \partial_n - i \tilde{A}_n )\tilde{a}_1 \right\|_{L^{2}(U_0)} \leq C (\lambda^2 \delta + \lambda \delta^{\frac{1}{2}} + \lambda^{-1}).
$$
Choose $\lambda=\delta^{-\frac{1}{4}}$ to minimize the right-hand side. Then
\begin{equation} \label{a1normal}
\left\| ( \partial_n - iA_n )a_1 - ( \partial_n - i \tilde{A}_n )\tilde{a}_1 \right\|_{L^{2}(U_0)} \leq C \delta^{\frac{1}{4}}.
\end{equation}
Consider the transport equation \eqref{transportaj} for $a_1$. 
In the semi-geodesic coordinates this equation takes the form
\begin{equation} \label{transporta1}
2i\xi_n \left( \partial_n - iA_n \right) a_1 = - \partial^2_n a_0 +  q + O(\delta^{\frac{1}{2}}),
\end{equation}
where $O(\delta^{\frac{1}{2}})$ represents the stably determined terms of order $O(\delta^{\frac{1}{2}})$. (In fact, $a_1=0$ in these expressions by the boundary condition in \eqref{transportaj}, but it is left here for the convenience of tracking the corresponding terms.) From the estimates \eqref{xinstability} \eqref{gnormal} and \eqref{transporta1} it follows that
\begin{equation} \label{difference1}
( -\partial^2_n a_0 + \partial^2_n \tilde{a}_0 ) 
%+ \left( ig^{\alpha\beta}\frac{\partial A_\alpha}{\partial x^\beta} - i \tilde{g}^{\alpha\beta}\frac{\partial \tilde{A}_\alpha}{\partial x^\beta} \right) 
+ (q -\tilde{q} ) = O(\delta^{\frac{1}{4}}).
\end{equation}
To obtain an expression of $\partial^2_n a_0$, we differentiate the transport equation in \eqref{transporta0} and evaluate it on $U_0$:
%\HOX{\TODO: We need an gauge transformation to change $A_n$ to be $0$ near $\partial M$. Then the term $i\xi_n\frac{\partial A_n}{\partial x^n}a_0$ will disappear.}
\begin{align*}
\partial^2_n a_0 = & -\frac{1}{4\det g} \partial^2_n \det g - \frac{1}{2 \xi_n} \partial^3_n \phi + \frac{i}{\xi_n} g^{\alpha\beta} \partial_n A_\alpha \xi_\beta + O(\delta^{\frac{1}{2}}) \\
= & -\frac{1}{4\det g} \partial^2_n \det g + \frac{1}{4\xi^2_n} \partial^2_n g^{\alpha\beta} \xi_\alpha\xi_\beta + \frac{i}{\xi_n} g^{\alpha\beta} \partial_n A_\alpha \xi_\beta + O(\delta^{\frac{1}{2}}),
\end{align*}
where the $O(\delta^{\frac{1}{2}})$ terms are estimated by \eqref{bdryg} and \eqref{gA} and we have used that $\partial_n A_n (x',0) = 0$ in \textbf{(M3)}. Inserting this into \eqref{difference1} and separating the even and odd parts in $\xi'$ gives (notice that $\xi_n=\sqrt{-g^{\alpha\beta}}\xi_\alpha\xi_\beta$ is an even function of $\xi'$.):
\begin{align} \label{difference2}
\left( \frac{1}{4\det g} \partial^2_n \det g - \frac{1}{4\det \tilde{g}} \partial^2_n \det \tilde{g} - \frac{1}{4\xi^2_n} \partial^2_n g^{\alpha\beta} \xi_\alpha\xi_\beta + \frac{1}{4 \tilde{\xi}^2_n} \partial^2_n \tilde{g}^{\alpha\beta} \xi_\alpha\xi_\beta \right) \nonumber \\
 + (q -\tilde{q} ) = O(\delta^{\frac{1}{4}}).
\end{align}
%\HOX{Notice here the term involving $\frac{\partial A_n}{\partial x^n}$ has been dropped.}
\begin{equation} \label{difference3}
- \frac{i}{\xi_n}g^{\alpha\beta} \partial_n A_\alpha \xi_\beta + \frac{i}{\tilde{\xi}_n}\tilde{g}^{\alpha\beta} \partial_n \tilde{A}_\alpha \xi_\beta  = O(\delta^{\frac{1}{4}}).
\end{equation}

To deal with \eqref{difference3}, we multiply the two terms by $\xi_n$ and $\tilde{\xi}_n$ respectively. This is valid since $\xi_n$ is stably determined in \eqref{xinstability}. Then applying Lemma \ref{mainlemma} shows
$$
\left\| \partial_n A_\alpha - \partial_n \tilde{A}_\alpha \right\|_{L^{2}(U_0)} \leq C \delta^{\frac{1}{2}}.
$$

To deal with \eqref{difference2}, recall the following matrix identity which is valid for any invertible matrix $S$
$$
\partial \log |\det S| = \text{tr} (S^{-1} \partial S).
$$
Taking $S=g^{\alpha\beta}$ and applying $\partial^{j-1}_n$ we see that 
$$
\partial^j_n \log (-\det g^{\alpha\beta}) = \partial^{j-1}_n (g_{\alpha\beta} \partial_n g^{\alpha\beta}), \quad\quad j=1,2,\dots
$$
For $j=2$, it gives
$$
g_{\alpha\beta} \partial^2_n g^{\alpha\beta} = \partial^2_n \log (-\det g^{\alpha\beta}) - \partial_n g_{\alpha\beta} \partial_n g^{\alpha\beta}.
$$
The right-hand side is stably determined by \textbf{(M2)} and \eqref{gnormal}, we thus get on $U_0$ that
\begin{equation} \label{difference4}
g_{\alpha\beta} \partial^2_n g^{\alpha\beta} - \tilde{g}_{\alpha\beta} \partial^2_n \tilde{g}^{\alpha\beta} = O(\delta^{\frac{1}{2}}).
\end{equation}
On the other hand, remember that the two metrics $g$ and $\tilde{g}$ have been modified to satisfy \textbf{(M2)}, thus by \eqref{estimate1}
$$
\frac{1}{4 \det g} \partial^2_n \det g - \frac{1}{4 \det \tilde{g}} \partial^2_n \det \tilde{g} = \left( \frac{1}{4 \det g} - \frac{1}{4 \det \tilde{g}} \right) \partial^2_n \det g  = O(\delta).
$$
This together with \eqref{difference2} gives
\begin{equation} \label{difference5}
\left( - \frac{1}{4\xi^2_n} \partial^2_n g^{\alpha\beta} \xi_\alpha\xi_\beta + \frac{1}{4 \tilde{\xi}^2_n} \partial^2_n \tilde{g}^{\alpha\beta} \xi_\alpha\xi_\beta \right) + (q -\tilde{q} ) = O(\delta^{\frac{1}{4}}).
\end{equation}
Again we multiply the terms without the tilde by $\xi^2_n$ and those with it  by $\tilde{\xi^2_n}$, using \eqref{xin} we have
$$
(\partial^2_n g^{\alpha\beta} + 4 q g^{\alpha\beta} - \partial^2_n \tilde{g}^{\alpha\beta} - 4 \tilde{q} \tilde{g}^{\alpha\beta} )\xi_\alpha \xi_\beta = O(\delta^{\frac{1}{4}}).
$$
Lemma \ref{mainlemma} claims
$$
(\partial^2_n g^{\alpha\beta} + 4 q g^{\alpha\beta} ) - ( \partial^2_n \tilde{g}^{\alpha\beta} + 4 \tilde{q} \tilde{g}^{\alpha\beta} ) = O(\delta^{\frac{1}{4}}).
$$
Multiplying those terms without $\tilde{ }$ by $g_{\alpha\beta}$, those with $\tilde{ }$ by $\tilde{g}_{\alpha\beta}$, then summing up in $\alpha, \beta$ yields
$$
( g_{\alpha\beta} \partial^2_n g^{\alpha\beta} + 4nq ) - ( \tilde{g}_{\alpha\beta} \partial^2_n \tilde{g}^{\alpha\beta} + 4n\tilde{q} ) = O(\delta^{\frac{1}{4}}).
$$
From \eqref{difference4} we come to the conclusion that
$$
\|q-\tilde{q}\| _{L^{2}(U_0)} \leq C \delta^{\frac{1}{4}}.
$$
Inserting this into \eqref{difference5} and applying Lemma \ref{mainlemma} establishes
$$
\|\partial^2_n g^{\alpha\beta}-\partial^2_n \tilde{g}^{\alpha\beta}\| _{L^{2}(U_0)} \leq C \delta^{\frac{1}{4}}.
$$

Putting the estimates on $g,A,q$ together, we have established 
$$
\|\partial^2_n g -\partial^2_n \tilde{g}\| _{L^{2}(U_0)} + \left\| \partial_n A_\alpha - \partial_n \tilde{A}_\alpha \right\|_{L^{2}(U_0)} + \|q-\tilde{q}\| _{L^{2}(U_0)} \leq C \delta^{\frac{1}{4}}.
$$
As before, interpolation and the Sobolev embedding theorem lead to
$$
\|\partial^2_n g -\partial^2_n \tilde{g}\| _{C^{m}(\overline{U}_0)} + \left\| \partial_n A_\alpha - \partial_n \tilde{A}_\alpha \right\|_{C^{m}(\overline{U}_0)} + \|q-\tilde{q}\| _{C^{m}(\overline{U}_0)} \leq C \delta^{\frac{\mu}{4}}
$$
for $m>0$ and $\mu<1$.
Repeating this type of argument will establish the stability for higher order derivatives of $g,A,q$ on $U_0$.
\end{proof}

\section{Interior Stability}
\subsection{$\Lambda_{g,A,q}^{\rm gl}$ recovers the lens relation $\mathcal{L}$ in a stable way} 

\begin{thm}\label{thm1b}
Under the assumptions in the Introduction, 
 $\Lambda_{g,A,q}^{\rm gl}$ is an elliptic FIO of order $1$ associated with the (canonical) graph of $\L$. 
\end{thm}

Note that we excluded lighlike covectors in $\WF(f)$. This excludes bicharacteristics (geodesics) tangent to $\bo$ carrying singularities of $u$. This is where the two Lagrangians (one of them being the diagonal) intersect. We also restricted $u$ to the first reflection and shortly after that. Without that, the canonical relations would contain powers of $L$. The theorem is a direct consequence of the geometric optics construction and propagation of singularities results for the wave equation and can be considered as  essentially known. 

As a consequence of Theorem~\ref{thm1b}, for every $s$, $\Lambda_{g,A,q}^{\rm loc}$ maps $H^s(U)$ into $H^{s-1}(U)$ and $\Lambda_{g,A,q}^{\rm gl}$ maps  $H^s(U)$ into   $H^{s-1}(V)$. Fixing $s=1$, one may conclude that the natural norms for those two operators are the $H^1\to L^2$ ones. While both operators are bounded in those norms, their dependence on the metric $g$ is not necessarily continuous if we stay in those norms. For $\Lambda_{g,A,q}^{\rm loc}$, we will see that the principal symbol (and the whole one, in fact) depends continuously on $g$; and in fact the whole  operator does, as well. On the other hand,  while the canonical relation of $\Lambda_{g,A,q}^{\rm gl}$ depends continuously on $g$, the operator itself does not. This observation was used in \cite{BaoZhang}, see also \cite{SUV-DNmap2014} for a discussion. 
 
\begin{proof}[Proof of Theorem~\ref{thm1b}] 
 We are still looking for a solution of the form \r{ansatz-h} with $f$ having a wave front set as in Theorem~\ref{thm1a}: in a conic neighborhood of a FP timelike $(x',\xi^0{}')$. 
Past pointing codirections can be handled the same way.  The solution is the same but we are now trying to extend it as far as possible away from $\bo$. We know that mircolocally, $u$ is supported in a small neighborhood of the null bicharacteristic (projecting to a null geodesic on $M$) issued from $(x_0,\xi^0)$ with $\xi^0$ future pointing with a projection ${\xi^0}'$ on the boundary, i.e., $\xi^0 = ({\xi^0}', \xi_n(x_0,\xi^0))$, where $\xi_n$ is given by \r{xin}, see Figure~\ref{DN_Lorentz_fig1}. This follows from the general propagation of singularities theory but in this particular case it can be derived from the fact that $T$ in \r{transportoperator} has its principal part a vector field along such null geodesics; and $\WF(u)$ can be analyzed directly with the aid of \r{ansatz-h}. 

Such a solution is guaranteed to exist only near some neighborhood of $x_0$ because the eikonal equation may not be globally solvable. On the other hand, the solution is still a global FIO applied to the boundary data $f$. It can also be viewed as a superposition of a finite number of local FIOs, each one having a representation of the kind \r{ansatz-h}. Indeed, we construct $u$ first near $\bo$. Then we restrict it to a timelike surface intersecting the null geodesic issued from $(x_0,\xi^0)$, and we chose that surface so that the geometric optics construction is still valid. We take the boundary data there, and solve a new similar problem, etc. By compactness arguments, we can cover the whole null geodesics until it hits $\bo$ again. We use this argument several times below. 

We will analyze first the map $F:f\mapsto u|_{S}$, where $S$ is a timelike surface as above, and the  \r{ansatz-h} is valid all the way to it, and a bit beyond it. 
%what happens if \r{ansatz-h} is valid all the way until that null geodesic hist $\bo$ again in $V$, and also in a small neighborhood of it. After that, we will use the arguments above. 
%Since the boundary condition $f$ vanishes in $V$, we need to construct a reflected solution that we subtract from $u$ in 

Change the coordinates $x$ so that $S=\{x^n=1\}$. This can be done if $S$ is close enough to $\bo$. 
Then \r{ansatz-h} with $x=(x',1)$ is a local representation of the FIO 
$F$ and its canonical relation is given by (see, e.g., \cite[Ch.~VIII]{Taylor-book2})
\[
(\nabla_{\xi'} \phi|_{x^n=1},\xi')\quad \longmapsto\quad  (x',\nabla_{x'}\phi|_{x^n=1}).
\]
By \r{Ham},  with the momentum $p$ projected to $T^*\{x^n=1\}$, we get that this is the lens relation $\mathcal{L}_1$ from $\bo$ to $S$ (instead of the image being on $\bo$ again). 

We can repeat this finitely many steps by choosing $S_1$, $S_2$, etc., to get a composition of finitely many canonical relations, starting with $\mathcal{L}_1$, then $\mathcal{L}_2$ maps data on $S_1$ to $S_2$, etc. That composition of, say $m$ of them, gives the lens relation from $\bo$ to $S_m$. In the final step, we need to take a normal derivative and reflect the solution as in \r{u_inc}   below. This would not change the projection of $\nabla_x\phi$ on the boundary. This completes the proof.
\end{proof}
%In this part we establish the following interior stability under some a-priori closeness conditions. 

To prove stable recovery of the lens relation $\mathcal L$, we recall that the $H^1\to L^2$ norm of the DN maps  is not suitable for measuring how close the canonical relations $\mathcal L$ and $\mathcal{\tilde{ L}}$ of the FIOs $\Lambda_{g,A,q}^{\rm gl}$ and $\Lambda_{\tilde g,\tilde A,\tilde q}^{\rm gl}$ are. Instead, we formulate stability based on  measuring propagation of singularities. Given a properly supported \PDO\ $P$  on $\bo$ near $(y_0,\eta^0)$, with a principal symbol $p_0$, we consider $\Lambda^* P\Lambda$, where $\Lambda = \Lambda_{g,A,q}^{\rm gl}$. By the Egorov theorem, this is actually a \PDO\ near $(x_0,\xi_0)$ with a principal symbol $(p_0\circ\mathcal L)\lambda_0$, where $\lambda_0$ is the principal symbol of $\Lambda \Lambda^*$ which depends on $g$. In this way, we do not recover $\mathcal{L}$ directly; instead we recover functions of $\mathcal L$ for various choices of $p_0$, multiplied by $\lambda_0$. Choosing a finite number of $P$'s  satisfying some non-degeneracy assumption, we can apply the Implicit Function Theorem to recover $\mathcal L$ locally. In fact, we choose below the  differential operators 
\be{Pj}
\{P_j\} = \{1, y^0,\dots,y^{n-1}, \partial/\partial{y^0}, \dots, \partial/\partial{y^{n-1}}\}.
\ee

%Choose a future pointing vector field $Z$ near $(y_0,\eta^0)$ and assume that the latter is normalized so that it belongs to $\Sigma := \{\langle \eta,Z\rangle <0\}$, where $\langle Z,\eta\rangle $ is the action of the covector $\eta$ (at $x$) on the vector $Z$. By the homogeneuity of $\mathcal L$, it is enough to restrict it locally to such a $\Sigma$. 

%\begin{thm}\label{thm_LR_stab}
%Let $p_j(x,\xi)$, $j=1,\dots 2n-2$ be symbols of order zero so that the map $\Sigma\ni (y,\eta)\to \{p_j\}$ is a local diffeomorphism near $(y_0,\eta^0)$, and set $p_0= |\xi|_g$.  Let 
%\[
%\delta := \sum_{j=0}^{2n-2} \left\| \Lambda^* P\Lambda -\tilde \Lambda^* P\tilde \Lambda\right\|_{L^2(U)\to H^2(U)}
%\]
%with $\Lambda := \Lambda_{ g, A, q}^{\rm gl}$, $\tilde \Lambda := \tilde \Lambda_{\tilde g,\tilde A,\tilde q}^{\rm gl}$. Assume that  $(g,A, q)$ and $(\tilde{g},\tilde  A,  \tilde{q})$ are $\epsilon$--close to a fixed tripple  $(g_0,A_0,q_0)$  in a certain  $C^{k}$ norm in the semi-geodesic normal coordinates near $x_0$ and near $y_0$. % with a constant $C>0$ depending on that bound with $k=k(m,\mu)$.  
%Then there exists $\eps>0$, $k>0$ and $\mu\in (0,1)$ so that 
%\be{QF}
%|(\mathcal{L} - \tilde{\mathcal{L}})(x,\xi')| \le C\delta^\mu\sqrt{- g(\xi',\xi')}, \quad \forall (x,\xi')\in\mathcal{U},
%\ee
%if $\mathcal{U}$ is small enough. 
%\end{thm}

\begin{thm}\label{thm_LR_stab}
Let   $(y^0,\dots, y^{n-1})$  be local coordinates on $\bo$ near  $y_0$. 
   Let 
\be{53}
\begin{split}
\sum_{j=0}^{n-1} \left\| \Lambda^* y^j\Lambda -\tilde \Lambda^* y^j\tilde \Lambda\right\|_{H^2(U)\to L^2(U)}&\le\delta, \quad 
\left\| \Lambda^* \Lambda -\tilde \Lambda^* \tilde \Lambda\right\|_{H^2(U)\to L^2(U)}\le\delta,\\
\sum_{j=0}^{n-1} \left\| \Lambda^* \frac{\partial}{\partial y^j}\Lambda -\tilde \Lambda^* \frac{\partial}{\partial y^j}\tilde \Lambda\right\|_{H^3(U)\to L^2(U)}&\le\delta,
\end{split}
\ee
with $\Lambda := \Lambda_{ g, A, q}^{\rm gl}$, $\tilde \Lambda := \tilde \Lambda_{\tilde g,\tilde A,\tilde q}^{\rm gl}$. Assume that  $(g,A, q)$ and $(\tilde{g},\tilde  A,  \tilde{q})$ are $\epsilon$--close to a fixed tripple  $(g_0,A_0,q_0)$  in a certain  $C^{k}$ norm in the semi-geodesic normal coordinates near $x_0$ and near $y_0$. % with a constat $C>0$ depending on that bound with $k=k(m,\mu)$.  
Then there exist $k>0$ and $\mu\in (0,1)$ so that 
\be{QF}
|(\mathcal{L} - \tilde{\mathcal{L}})(x,\xi')| \le C\delta^\mu\sqrt{- g(\xi',\xi')}, \quad \forall (x,\xi')\in\mathcal{U},
\ee
if $\mathcal{U}$ and $\epsilon>0$ are small enough. 
\end{thm}

A few remarks: 
\begin{enumerate}[{(a)}] 
\item 
 The square root term is just a homogeneity factor. 
\item  The cotangent bundle $T^*\bo$ is not a linear space, therefore the difference $\mathcal{L} - \tilde{\mathcal{L}}$ makes sense in fixed coordinates only.  
\item The norms in \r{53} are the natural one since the operators we subtract there are \PDO s of order two and three, respectively. 
\item  The norms in \r{53} are equivalent to studying the quadratic form $(\Lambda f,P_j\Lambda f)- (\tilde \Lambda f,P\tilde \Lambda f)$. 
%\item  In local coordinates $(y^0,\dots, y^{n-1}, \eta^0,\dots\eta^{n-1})$  on $T^*\bo$ in which $(0,\dots,0,1)$ is timelike and covectors with $\eta_n=0$ are spacelike,  one can   simply take the symbols $1$, $(x^0,x^1,\dots, \allowbreak x^{n-1})$ and $(\eta_1/\eta_0 , \dots, \eta_{n-1}/\eta_0) $ as $\{p_j\}$. 

\item One could reduce the number of the $P_j$'s to $2n-2$; in fact, $P_0=1$ in \r{53} is not needed, as it follow from Remark~\ref{rem2}, since we can recover $\eta'/\eta_n$ and use the fact $\eta=(\eta',\eta_n)$ is a null covector. 
\end{enumerate}

We prove  Theorem~\ref{thm_LR_stab} in the end of this section. 

\subsection{Stable recovery of the light ray transforms of $A$ and $q$} 
Let, as in the Introduction, $\xi^0\in T_{x_0}M\setminus  0$ be the future-pointing lightlike co-vector whose projection on $T^\ast\partial M\setminus 0$ is the timelike co-vector $\xi^0{}'$ as in the definition of the semi-global DN map. Let $\gamma_0:= \gamma_{x_0,\xi^0{}'}$ be the lightlike geodesic issued from $(x_0,{\xi^0})$ which intersects $\partial M$ at another point $y_0$. 
%Denote by $U_0$ the neighborhood of $x_0$ as in Theorem \ref{bdrystability}, then the semi-geodesic normal coordinate $(x',x^n)$ exists near $U_0$ in $M$. By translating $U_0$ along the lightlike geodesics issued near $(x_0,\hat{\xi}^0)$ we obtain a neighborhood of $\gamma_{x_0,\hat{\xi}^0}$ in $M$, say $\Omega$, and a neighborhood of $y_0$ in $\partial M$, say $V_0$. 
Let $V$ be a neighborhood of  $y_0$ containing all endpoints of future pointing geodesics issued  from $\bar{\mathcal{U}}$.  
Choose and fix any parameterization of the lighlike geodesics close to  $\gamma_0$   by normalizing  $\xi{}'$. This defines a  hypersurface $\mathcal{U}_0$ in $\mathcal{U}$. 
The theorem below holds if $\mathcal{U}$ is a small enough neighborhood of $(x_0,\xi^0{}')$,  and therefore $\mathcal{U}_0$ is small enough, as well. Then $L_1$ and $L_0$ are well defined on $\mathcal{U}_0$.

\begin{thm} \label{intstability} Fix a Lorentzian metric $g$, and  $(x_0,\xi^0)$ satisfying the assumptions above. Let $(A,q)$ and $(\tilde A,\tilde q)$ be two pairs of magnetic and electric potentials. 
 Denote $\delta:=\|\Lambda_{g,A,q}^{\rm gl}-\Lambda_{ g,\tilde A,\tilde q}^{\rm gl}\|_{H^1(U)\to L^2(V)}$. Then  
%in the semi-geodesic normal coordinates near $x_0$.

%there exists $k\gg 1$ such that the estimates
\begin{enumerate}
	\item[(a)]  for any $\mu<1$ and $m\geq 0$, the following estimates are valid for some integer $N$ whenever $g, A, \tilde{A}, q, \tilde{q}$ are bounded in a certain $C^{k}$ norm
\[
\|L_1(A-\tilde A)-2\pi N \|_{C^m(\bar{\mathcal{U}}_0)} \leq C\delta^\mu.
\]

	\item[(b)] Under the a-priori condition $\|A-\tilde{A}\|_{C^1(\bar{\Omega})}\leq \delta_1$ for some $\delta_1>0$, %the local light ray transform of $q-\tilde{q}$ is stably determined in the sense that
	for any $0<\mu<\mu'$ and $m\geq 0$, the following estimate is valid whenever $g, A, \tilde{A}, q, \tilde{q}$ are bounded in a certain  $C^{k}$ norm
	$$
\| L_0 (q-\tilde{q}) \|_{C^m(\bar{\mathcal{U}_0})}  \leq C (\delta^\mu+ \delta_1^\mu).
  $$
\end{enumerate}
\end{thm}

If there are no conjugate points along $\gamma_0$, the proof can be done using a geometric optics construction of the kind \r{u_N} but with a different phase in \r{u_N} all the way along that geodesic and taking the normal derivative in $V$. Since we do not want to assume the no-conjugate points assumption, we will  proceed in a somewhat different way. 

The fact that we cannot rule out the case $N\not=0$ based on those arguments can be considered as a manifestation of the Aharonov-Bohm effect. If $\tilde A$ and $A$ are a priori close, then $N=0$. 

We start with a preparation for the proof of the theorem. 
Consider first the geometric optics parametrix of the kind \r{ansatz-h} of the outgoing solution $u$ like in the previous section. We assume that  the boundary condition $f$ has a wave front set in the timelike cone on the boundary, and for simplicity, assume that it is in the future pointing one ($\tau<0$ in local coordinates for which $\partial/\partial t$ is future pointing).  Assume at this point that the construction is valid in some neighborhood of the maximal $\gamma_0$. We microlocalize all calculations below there. All inverses like $D^{-1}$, etc., below are microlocal parametrices and the equalities between operators are modulo smoothing operators in the corresponding conic microlocal neighborhoods  depending on the context.

The construction is the same  to that in the previous section, but this time the outgoing solution $u$ is constructed near the bicharacteristic issued from $(x_0,\xi^0{}')$   all the way to $y_0$. Since the solution can reach the other side of the boundary, we need to reflect it at the boundary to satisfy the zero boundary condition. We write the solution $u$ as the sum of the incident wave $u^{\rm inc}$ and the reflected wave $u^{\rm ref}$: $u=u^{\rm inc}+u^{\rm ref}$ where
\be{u_inc}
\begin{split}
u^{\rm inc}(x) &=  (2\pi)^{-n}\int e^{\i\phi(x,\xi')}  \left( a^{\rm inc}_{0}+   a^{\rm inc}_{1} + R^{\rm inc} \right) \! (x,\xi')       \hat f(\xi')\, \d\xi',\\
u^{\rm ref}(x) &= (2\pi)^{-n}\int e^{\i\phi^{\rm ref} (x,\xi')}  \left( a^{\rm ref}_{0}+   a^{\rm ref}_{1} + R^{\rm ref} \right) \! (x,\xi')       \hat f(\xi')\, \d\xi'. 
\end{split}
\ee
Here the phase function $\phi^{\rm ref}$ solves the same eikonal equation as $\phi$ does but satisfies the boundary condition $\phi^{\rm ref}|_{V}=\phi$. It differs from $\phi$ by the sign of its normal derivative $\partial \phi^{\rm ref}/ \partial x^n = -\partial \phi^{\rm ref}/ \partial x^n>0$ on $x^n=0$. 
%We may choose $V$ to be the image of $U$ under translations by all lightlike geodesics issued near $(x_0,\hat{\xi}^0)$. 
The amplitudes are of order $0$ and $-1$, respectively, and  satisfy
\begin{align*}
T^{\rm inc} a^{\rm inc}_{0} = 0, & \quad\quad a^{\rm inc}_{0}|_{U} = \chi, \\
T^{\rm ref} a^{\rm ref}_{0} = 0, & \quad\quad a^{\rm ref}_{0}|_{V} = -a^{\rm inc}_{0}|_{V}, \\
iT^{\rm inc} a^{\rm inc}_{1} = -P a^{\rm inc}_{0}, & \quad\quad a^{\rm inc}_{1}|_{U} = 0, \\
iT^{\rm ref} a^{\rm ref}_{1} = -P a^{\rm ref}_{0}, & \quad\quad a^{\rm ref}_{1}|_{V} = - a^{\rm inc}_{1}|_{V},
\end{align*} 
where $T^{\rm inc}$ and $T^{\rm ref}$  are the transport operators defined in \eqref{transportoperator}, related to the corresponding phase function, 
and the remainder terms  are of order  $-2$.
%$$
%\|R^{\rm inc}\|_{What space?} + \|R^{\rm ref}\|_{What space?} \leq \frac{C}{\lambda^2}.
%$$
%\HOX{The spaces of $R^{\rm inc}$ and $R^{\rm ref}$ depend on the spaces in the lacking existence result on Page 1.}

%By shrinking  $V_0$ if necessary, we may arrange  that $a^{\rm inc}_{0} \geq1/C >0$ on $V_0$.
%By \r{T}, 
Replace $A$ and $\tilde A$ with their gauge equivalent field satisfying (M3) on $V$. This does not change their light ray transforms. 
  A direct computation, which can be justified as \r{DNmap}, yields
\begin{equation} \label{Lambda_f}
\begin{split}
\Lambda^{\rm gl}_{g,A,q}f &= (2\pi)^{-n}\int e^{\i \phi(x,\xi')} \left( 2\i (\partial_\nu \phi) a^{\rm inc}_{0}  2\i (\partial_\nu \phi) a^{\rm inc}_{1} + \partial_\nu (a^{\rm inc}_{0} + a^{\rm ref}_{0})+ a_{-1} \right) \hat f(\xi')\,\d\xi' ,
\end{split}
\end{equation}
where $a_{-1}$ is of order $-1$ and $\phi$ and the amplitudes are restricted to $x\in V$. 

The expression \r{Lambda_f} allows us to  factorize $\Lambda_{g,A,q}^{\rm gl}$ as $\Lambda_{g,A,q}^{\rm gl}= 2 N_0D$ modulo FIOs of order $0$ associated with the same canonical relation,   where  $Df$ is the trace of $u^{\rm inc}$ on $V$ (a ``Dirichlet-to-Dirichlet map'') and $ N_0$ is the DN map $\Lambda_{g,0,0}^{\rm loc}$ but localized in $V$. Note that replacing $A$ and $q$ in $N_0$ by zeros or not contributes to lower order error terms. 
Let $D_0$ be the operator $D$ related to $A=0$, $q=0$. Let $ N_0^{-1}$ and $D_0^{-1}$ be microlocal parametrices of those operators which are actually parametrices of the local Neumann-to-Dirichlet map and the incoming Dirichlet-to-Dirichlet one from $V$ to $U$. Then
\be{d1}
  D_0^{-1} N_0^{-1}  \Lambda^{\rm gl}_{g,A,q} = 2 D_0^{-1}D \quad \text{mod $S^{-1}$}
\ee
is a \PDO\ of order $0$. 

In the next lemma, we do not assume that the geometric optics construction is valid along the whole $\gamma_0$. 
\begin{lemma}\label{lemma6}
The operator   $ D_0^{-1}N_0^{-1} \Lambda^{\rm gl}_{g,A,q} $ is a \PDO\ of order zero   in $\mathcal{U}$ with principal symbol 
\be{l6}
2\exp\left\{ \i L_1A (\gamma_{x',\xi'}) \right\},
\ee
%where $L (\Box_g\phi) (x',\xi')= \int (\Box_g\phi) (\gamma_{x',\xi'}(s), \dot\gamma_{x',\xi'}(s))\, d s$, and 
where $\gamma_{x',\xi'}$ is the future pointing lightlike geodesic issued from $x'$ in direction $\xi$ with projection $\xi'$. 
\end{lemma}

\begin{proof}
By \r{d1}, we need to find the principal symbol of $ D_0^{-1}D$. 

The transport equation for $a^{\rm inc}_{0}$ is
$$
\left[2g^{jk} (\partial_j \phi) (\partial_k - \i A_k ) + \Box_g \phi \right] a^{\rm inc}_{0} = 0, \quad\quad a^{\rm inc}_{0}|_{U}=1.
$$
As explained right after \r{Ham}, $g^{jk} (\partial_j \phi)\partial_k$ is the tangent vector field along 
the null geodesic $\gamma_{x',\xi'}$.  Therefore, with $\Gamma(s) := (\gamma_{x',\xi'}(s), g\dot \gamma_{x',\xi'}(s))$, as before, on the set $\chi=1$ we get $a^{\rm inc}_{0}=\mu$, see \r{mu}, i.e., 
\begin{equation} \label{azeroinc}
\begin{split}
a^{\rm inc}_{0}(\Gamma(s) ) &=  \exp\bigg\{-\frac{1}{2} \int_0^s (\Box_g\phi) (\Gamma(\sigma) )\, \d \sigma \bigg\}\\
& \qquad \times\exp\bigg\{\i \int_0^s A_k \circ\gamma_{x',\xi'}(\sigma) \dot \gamma^k_{x',\xi'}(\sigma)  \,\d \sigma\bigg\}.
\end{split}
\end{equation}
Take $s=s(x,\xi)$ so that $\gamma_{x',\xi'}(s)\in V$ to get 
\[
(a_0^{\rm inc}\circ\mathcal{L})(x',\xi') = \exp\left\{-\frac12 L(\Box_g\phi)(x',\xi') +\i L_1A(x',\xi') \right\},
\]
where we use the coordinates $(x',\xi')$ to parameterize the lightlike geodesics locally, and the definition of  $L(\Box_g\phi)$ is clear from \r{azeroinc}. 

To construct a representation for $ D_0^{-1}$, note first that when $A=0$, the term involving $L_1A$ is missing above. We look for a parametrix of the incoming solution of $\Box_gu=0$ with boundary data $u=h$ on $V$ with $\WF(h)\subset \mathcal{V}$ of the form  
\be{ansatz-h0}
u (x)= (2\pi)^{-n}\int e^{\i\phi(x,\xi')} b(x,\xi') \hat f(\xi')\, \d\xi',
\ee
where $\phi$ is the same phase as in the first equation in \r{u_inc} and $f$ (not related to \r{f}) depending on $h$ as below. The amplitude $b$ solves the transport equation along the same bicharacteristics (with different coefficients since $A=0$, $q=0$) with the initial condition:
\[
b|_V = a^{\rm inc}|_V,
\]
where $a^{\rm inc}$ is the full amplitude in the first equation in \r{u_inc}. Restricted to $V$, the map $f\to u|_V$ is just $Df$. 
%\[
%Fh_1 := (2\pi)^{-n}\int e^{\phi_V(x,\eta')} b_V(x,\eta') \hat h_1(\eta')\, \d\eta',
%\]
%where $\phi_V$ and $b_V$ are the traces of $\phi$ and $b$ on $V$, respectively. 
Then to satisfy $u=h$ on $V$, we need to solve $Df=h$, i.e., to take $f=D^{-1}h$ microlocally. %On the other hand, we will apply \r{azeroinc} to $u^{\rm inc}$ which equals \r{azeroinc} on $\bo$. 

To illustrate the argument below better, suppose that we are solving the ODE
\[
y'+ay=0, \quad y(0)=1
\]
from $t=0$ to $t=1$, where $a=a(t)$. Then we solve
\[
y_1'+a_1y_1=0, \quad y_1(1)=y(1),
\]
where $a_1=a_1(t)$. A direct calculation yields
\[
y(t) = \exp\Big\{ -\int_0^t a(s)\, \d s \Big\}, \quad y_1(t) = \exp\Big\{ -\int_1^t a_1(s)\, \d s \Big\}\, y(1).
\]
In particular,
\[
  y_1(0) = \exp\Big\{ -\int_0^1 ( a_1(s)-a(s))\, \d s \Big\}.
\]

We apply those argument to the transport equation to get
\[
b|_U =  \exp\big\{ \i L_1A(\gamma_{x',\xi'})\big\}.
\]
Then 
\[
D_0^{-1}Df = (2\pi)^{-n}\int e^{x'\cdot\xi'}  \exp\big\{ \i L_1A(\gamma_{x',\xi'})\big\} \hat f(\xi')\, \d\xi'.
\]
This proves the lemma under the assumption that the geometric optics construction is valid in a neighborhood of  $\gamma_0$.

	To prove the theorem in the general case, let $S_1, \dots, S_k$ be small timelike surfaces intersecting $\gamma_0$ in increasing order, from $U$ to $V$ so that the geometric optics construction is valid in a neighborhood of each segment of $\gamma_0$ cut by two consecutive surfaces of the sequence $\{U,S_1,\dots,S_k,V\}$. This determines Dirichlet-to-Dirichlet maps $D_1$, from $U$ to $S_1$; then $D_2$, from $S_1$ to $S_2$, etc., until $D_{k+1}$ from $S_k$ to $V$. Then $D=D_{k+1}D_k\dots D_1$. Similarly, $ D_0= D_{0,k+1} D_{0,k}\dots D_{0,1}$. Then \r{d1} is still valid and takes the form 
\[
  D_0^{-1} N_0^{-1}  \Lambda^{\rm gl}_{g,A,q} = 2 D_{0,1}^{-1}\dots   D_{0,k}^{-1}  D_{0,k+1}^{-1} D_{k+1} D_k\dots D_1 \quad \text{mod $S^{-1}$}.
 \]
By  Lemma~\ref{lemma6}, $D_{0,k+1}^{-1} D_{k+1} $ is a \PDO\ on $V$ with principal symbol $\exp\{\i L_1^{(k+1)}A\}$, where $L_1^{(k+1)}$ is the light ray transform $L_1$ restricted to geodesics between $S_k$ and $V$. Then we apply    Egorov theorem, see \cite[Theorem~25.2.5]{Hormander4}, to conclude that 
$  D_{0,k}^{-1} \big( D_{0,k+1}^{-1} D_{k+1}\big) D_k$ is a \PDO\ with a principal symbol that of $D_{0,k+1}^{-1} D_{k+1} $, pulled back by $\mathcal{L}_{k+1}$, the canonical relation between $S_k$ and $V$, multiplied by the principal symbol of $  D_{0,k}^{-1}  D_k$. The result is then \r{l6} without the factor $2$ with the integration between $S_k$ (through $S_{k+1}$) to $V$. Repeating this argument several times, we complete the proof of the lemma. 
\end{proof}

\subsection{Stability of the light ray transform of the magnetic field}
\begin{proof}[Proof of Theorem~\ref{intstability}(a)] 
We have 
\be{sa1}
\begin{split}
\big\| D_0^{-1} N_0^{-1}&\big( \Lambda^{\rm gl}_{g,A,q} -   \Lambda^{\rm gl}_{g,\tilde A,\tilde q}\big)\big\|_{H^1(U)}\\
& \le 
C\big\|\Lambda_{g,A,q}^{\rm gl}-\Lambda_{ g,\tilde A,\tilde q}^{\rm gl}\big\|_{H^1(U)\to L^2(V)}=C\delta.
\end{split}
\ee
Set $R:=  D_0^{-1} N_0^{-1} \big(\Lambda^{\rm gl}_{g,A,q} -   \Lambda^{\rm gl}_{g,\tilde A,\tilde q}\big)$. By Lemma~\ref{lemma6}, $R$ is a \PDO\ in $\mathcal{U}$ of order $0$ with principal symbol
\[
r_0(x',\xi') = 2\exp\left\{\i L_1(\tilde A - A)(\gamma_{x',\xi'})\right\},
\]
and we have $\|R\|_{ H^1(V)}\le C\delta$, by \r{sa1}. We need to derive that $r_0$ is ``small'' in $\mathcal{U}$, as well. We essentially did that in the proof of Theorem~\ref{bdrystability}. Choose $f$ as in \r{f}.  By \cite[Ch.~VIII.7]{Taylor-book0}, on the set $\chi=1$,   $  e^{-i\lambda x'\cdot \xi'} R f$ is equal to the full symbol of $\Lambda_{g,A,q}^{\rm loc}$ with $\lambda=|\xi|$  and $\xi$ in \r{DNmap} bounded, say, unit.  Therefore,
\be{q0}
r_0(x',\xi') = e^{-i\lambda x'\cdot \xi'}Rf + O(1/\lambda)
\ee
in $C^k$ for every $k$. 
Since $\|f\|_{L^2}=C$ and $\|f\|_{H^1}\sim\lambda$, \r{sa1} yields
\[
\|r_0(\cdot,\xi')\|_{H^1(U)}\le C\lambda\delta +C/\lambda,
\]
uniformly for $\xi$ in some neighborhood of $\xi^0{}'$. 
With a little more efforts one can remove $\lambda$ from $C\lambda\delta$ but this is not needed. Take $\lambda=\delta^{-1/2}$ to get 
\[
\Big\|\exp\left\{\i L_1(\tilde A - A)(\gamma_{x',\xi'})\right\}\Big\|_{H^1(\bar{\mathcal{U}}')}\le C\delta^{1/2}.
\]
Using interpolation estimates, we can replace the $H^1$ norm by any other one at the expense of lowering the exponent on the right from $1/2$ to another positive one, if $k$ in Theorem~\ref{intstability} is large enough. Since $|e^{iz}-1|<\eps$ implies $|z-2\pi  N|<C\eps$ for some integer $N$, this proves part (a) of the theorem. 
\end{proof}

\subsection{Stability of the light ray transform of the potential}
\begin{proof}[Proof of Theorem~\ref{intstability}(b)] 
First, we will reduce the problem to the case $\tilde A=A$. For $\Lambda^{\rm gl}_{g,\tilde A,\tilde q}- \Lambda^{\rm gl}_{g,A,\tilde q}$, we get a representation as in \r{Lambda_f} with a principal symbol with seminorms $O(\delta_1^{\mu'})$, since we can use interpolation estimates to estimate the higher derivatives of $\tilde A-A$. 
%We will denote that amplitude by $O_1(\delta^{\mu'})$ in what follows with the subscript $1$ reminding us of the order and will assume, and will add it to assuming at the same time $\tilde A=A$. 
Apply a parametrix $\left(\Lambda^{\rm gl}_{g,A,\tilde q}\right)^{-1}$ to that difference to get a \PDO\ $Q$ of order $0$ microlocally supported in $\mathcal{U}$. If the geometric optics construction is valid all the way from $U$ to $V$, we get as in the proof of (a) that $Qf = O(\delta_1^{\mu'})+ O(1/\lambda)$ in $H^1$. This implies the same estimate for $\big\|\big(\Lambda^{\rm gl}_{g,\tilde A,\tilde q}- \Lambda^{\rm gl}_{g,A,\tilde q}\big)f\big\|_{L^2}$. 
In the general case, we can prove the same estimate as in the proof of (a). We will use this later and for now, we assume $\tilde A=A$.

\begin{lemma}\label{lemma7}
The operator   $   D^{-1} N_0^{-1}\left( \Lambda^{\rm gl}_{g,A,\tilde q} - \Lambda^{\rm gl}_{g,A,q}\right) $ is a \PDO\ of order $-1$ on $U$ with principal symbol 
\be{20}
%-2\i\left[ e^{ \i L_1A }L_0(\tilde q-q)\right]\circ \gamma_{x',\xi'} ,
2\left[ L_0(\tilde q-q)\right]\circ \gamma_{x',\xi'} ,
\ee
%where $L (\Box_g\phi) (x',\xi')= \int (\Box_g\phi) (\gamma_{x',\xi'}(s), \dot\gamma_{x',\xi'}(s))\, d s$, and 
where $\gamma_{x',\xi'}$ is the future pointing lightlike geodesic issued from $x'$ in direction $\xi$ with projection $\xi'$. 
\end{lemma}

\begin{proof}
Assume first that the geometric optics construction is valid in a neighborhood of the whole $\gamma_0$. In the amplitude 
\[
 -2\i (\partial_\nu \phi) a^{\rm inc}_{0} - 2\i (\partial_\nu \phi) a^{\rm inc}_{1} + \partial_\nu (a^{\rm inc}_{0} + a^{\rm ref}_{0})+ a_{-1} 
\]
in \eqref{Lambda_f},  the terms  $-2\i (\partial_\nu \phi) a^{\rm inc}_{0}$ and  $  \partial_\nu (a^{\rm inc}_{0} + a^{\rm ref}_{0})$ do not depend on $q$, see \r{azeroinc}. The other two terms depend on $q$ but they are of different orders. Therefore,
\be{sa2}
\begin{split}
\Big(\Lambda^{\rm gl}_{g, A,\tilde q}- \Lambda^{\rm gl}_{g,A,q}\Big)f& =  (2\pi)^{-n}\int e^{\i \phi(x,\xi')} \Big( - 2\i (\partial_\nu \phi) (\tilde a^{\rm inc}_{1} - a^{\rm inc}_{1} )  +(a_{-1}- \tilde a_{-1}) \Big) \hat f(\xi')\,\d\xi'|_V .
\end{split}
\ee
The order of the FIO above is zero. 
As in the previous proof, we can represent this as a composition of $2 N_0$ with the operator $\tilde D-D$ (the difference of two such Dirichlet-to-Dirichlet maps):
\be{a3aa}
\Lambda^{\rm gl}_{g, A,\tilde q}- \Lambda^{\rm gl}_{g,A,q}= 2 N_0(\tilde D-D)
\ee
modulo FIOs of order $-2$. 
That operator $\tilde D-D$ is an FIO with a symbol, compare with \r{sa2},
\be{a3a}
\sigma(\tilde D-D) = - 2\i (\tilde a^{\rm inc}_{1} - a^{\rm inc}_{1} )  +a_{-2},
\ee
with $a_{-2}$  of order $-2$. 

To compute $a^{\rm inc}_{1}$, recall  the transport equation for $a^{\rm inc}_{1}$ 
\be{a4}
\left[2g^{jk} \partial_j \phi (\partial_k-\i A_k ) + \Box_g \phi \right] a^{\rm inc}_{1} = \i Pa^{\rm inc}_{0}, \quad\quad a^{\rm inc}_{1}|_{U}=0
\ee
where
%\[
%iPa^{\rm inc}_{0}= i\Box_g a^{\rm inc}_{0} + \text{div}\, A^\sharp a^{\rm inc}_{0}  + \langle A, \nabla a^{\rm inc}_{0} \rangle_g -  i\langle A, A\rangle_g a^{\rm inc}_{0} + iqa^{\rm inc}_{0}
%\]
\[
\i Pa^{\rm inc}_{0}= \i P_{g,A,0} a^{\rm inc}_{0} + \i qa^{\rm inc}_{0}. 
\]
%with $A^\sharp$ the vector field obtained by raising indices.
The first term on the right is independent of $q$. By \r{T}, \r{mu}, with  $\Gamma(s)$ as in \r{azeroinc}, we get 
\be{a5}
\begin{split}
a^{\rm inc}_{1}\left(\Gamma(s)\right) &= \frac{\i a^{\rm inc}_{0}}{2} \int_0^s\frac1{a^{\rm inc}_{0}} \left[ P_{g,A,0} a^{\rm inc}_{0}+ qa^{\rm inc}_{0}\right]\circ\Gamma(\sigma) \,\d \sigma\\
& =\frac{\i a^{\rm inc}_{0}}{2} \int_0^s\Big[ \frac1{a^{\rm inc}_{0}}  P_{g,A,0} a^{\rm inc}_{0}+ q\Big]\circ\Gamma(\sigma) \,\d \sigma . 
\end{split}
\ee
The potential $q$ depends on $x$ only,  so $q\circ\Gamma(s)= q\circ\gamma(s)$. In \r{a5}, only the last term depends on $q$ and is an integral of $q$ over lightlike geodesics multiplied by an elliptic factor. Note that the integral, as well as $a^{\rm inc}_{1}$, are homogeneous of order $-1$ in $\xi'$, as they should be. 

We go back to \r{a3a} now. Using \r{a5},  the terms involving $P_{g,A,0}$ and $P_{g,\tilde A,0}$ cancel below and we get 
\be{a6}
\sigma(\tilde D-D)\circ \mathcal{L} = \i  a^{\rm inc}_{0}  L_0(\tilde q-q)   +a_{-2},
\ee
where $a_{-2}$ is a symbol of order $-2$, different form the one above. 

Similarly to \r{d1}, we have
\be{a7}
  D^{-1} N_0^{-1}\left( \Lambda^{\rm gl}_{g,A,\tilde q} - \Lambda^{\rm gl}_{g,A,q}\right) = 2 D^{-1}(\tilde D-D) \quad \text{mod $S^{-2}$}. 
\ee
%where $D_1$ is as in part (a) but related to $P_{g,A,q}$, in other words, $A$ is not zero and $q$ does not matter for the validity of \r{a7}. 
%To compute the principal symbol of the r.h.s.\ of \r{a7} 
%
%\[
%Fh_1 := (2\pi)^{-n}\int e^{\phi_V(x,\eta')} b_V(x,\eta') \hat h_1(\eta')\, \d\eta',
%\]
%where $\phi_V$ and $b_V$ are the traces of $\phi$ and $b$ on $V$, respectively. 
Therefore, we need to compute the principal symbol of $2D^{-1}(\tilde D-D)$. Let $R$ be a \PDO\ in $U$ with principal symbol $r_{-1}$ given by  \r{20}. Then, in $\mathcal{U}$, $DR$ is an FIO of the type \r{ansatz-h0} with $x\in V$ with the same phase function and a principal amplitude $b_0$ solving $Tb_0=0$, $b_0|_U  = r_{-1}$. By \r{T}, the solution restricted to $x\in V$ is given by $\mu r_{-1}\circ \mathcal{L}^{-1}|_{V}$. Recall that $\mu=a_0^{\rm inc}$.  By \r{a6}, this is $2\sigma(\tilde D-D)$ modulo symbols of order $-2$. Therefore, $DR= 2(\tilde D-D)$ modulo FIOs of order $-2$. This proves the lemma under the assumption that the geometric optic construction is valid along the whole $\gamma_0$. 

In the general case, we repeat the arguments of Lemma~\ref{lemma6}. We represent $ D$ and $\tilde D$ as a composition $ D=D_{k+1}\dots D_1$, and similarly for $\tilde D$. We will do the first step. Consider $2(D_2D_1)^{-1}(\tilde D_2\tilde D_1 -D_2D_1)$. We have
\[
\begin{split}
2(D_2D_1)^{-1}&(\tilde D_2\tilde D_1 -D_2D_1)\\
&= 2D_1^{-1}D_2^{-1}\left(  (\tilde D_2 - D_2)\tilde D_1 + D_2(\tilde D_1-D_1)\right)\\
& = D_1^{-1}R_2\tilde D_1 + R_1 = D_1^{-1}R_2 D_1 + R_1
\end{split}
\]
modulo FIOs of order $-2$, where $R_j =2D_j^{-1}(\tilde D_j-D_j)$, $j=1,2$.  We apply Egorov's theorem to $D_1^{-1}R_2 D_1 $ to conclude that it is a \PDO\ on $U$ with a principal symbol equal to the sum of two terms as in \r{20} with $L_0$ taken over the geodesic segments between $U$ and $S_1$ first, and $S_1$ and $S_2$ second. The sum is equal to \r{20} with $L_0$ taken over the union of those segments. Repeating this arguments to include $D_2$, etc., completes the proof of the lemma.
\end{proof}

We finish the proof of part (b) as we did that for part (a). Set $R=  D^{-1} N_0^{-1}\left( \Lambda^{\rm gl}_{g,A,\tilde q} - \Lambda^{\rm gl}_{g,A,q}\right)$. It is a \PDO\ of order $-1$ rather than of order $0$ as in (a). The analog of \r{sa1} is still true. If, as above, $r_{-1}$ is the principal symbol of $R$, then by Lemma~\ref{lemma7},
\[
r_{-1}(x',\xi') =  -2\i [L_0(\tilde q-q)]\circ\gamma_{x',\xi'} =\lambda e^{-\i\lambda x'\cdot\xi'}Rf+O(1/\lambda)
\]
with $f$ as in \r{f}, compare with \r{q0}. Then
\[
\|r_{-1}(\cdot,\xi')\|_{H^1(U)}\le C\lambda^2\delta+C/\lambda.
\]
Choose $\lambda=\delta^{-1/3}$ to get $\|r_{-1}(\cdot,\xi')\|_{H^1(U)}\le C\delta^{1/3}$. This completes the proof of the theorem. 
\end{proof}
\subsection{Proof of the stable recovery of  the lens relation}

\begin{proof}[Proof of Theorem~\ref{thm_LR_stab}]
We use the notation above. %We have, see the paragraph following \r{Lambda_f}, $\Lambda=2N_0D$ modulo lower order terms. 
Recall the remark preceding Theorem~\ref{thm_LR_stab} above. 
The operator $\Lambda^* P \Lambda$ is a \PDO\ with a principal symbol $(p_0\circ\mathcal{L})\lambda_0 $. 
%Then $\Lambda^* P \Lambda = 4D^*N_0^*PN_0D$ modulo lower order terms. The operator $N_0^*N_0$ is a \PDO\ with a principal symbol $\eta_n^2 = -g^{\alpha\beta}(y') \eta_\alpha \eta_\beta$, see \r{N0},  which is just the normal component  of $\eta$ squared. 
% The operator $D^*D$ is  identity microlocally. By Egorov's theorem, the principal symbol of $\Lambda^* P \Lambda$ is $(4\eta_n^2 p)\circ\mathcal{L}$, where $p$ is the principal symbol of $P$. 
Take $p = p_0=1$ as in \r{Pj} to recover $\lambda_0$ first. Knowing the latter, we recover $p_j\circ\mathcal{L}$ for $j=1,\dots 2n-1$, see \r{Pj}.  That gives us $(y,\eta')$ in \r{L} as functions of $(x,\xi')$. 
 Therefore, we reduce the stability problem to the following: show that the principal symbol of a \PDO\ $A$ of order $m$ is determined by $A:H^m\to L^2$ in a stable way which is resolved by the lemma below, see also \r{estimate0}, \r{xinstability}. Note that the lemma is a bit more general than what we need since $\{P_j\}$ are simple multiplication and differentiation operators. 
 
\begin{lemma}\label{lemma_P}
Let $A$ be  \PDO\ in $\R^n$ with kernel supported in $K\times K$, where $K\subset\R^n$ is compact. Let $p_m$ be its principal symbol homogeneous of order $m$. Then
\[
|p_m(x,\xi)|\le C|\xi|^{m}\|P\|_{H^m\to L^2} %\quad \text{for $|\xi|\ge1$}. 
\]
with $C>0$ depending on $K$ only. 
\end{lemma}
\begin{proof}
Take $f=e^{\i x\cdot\xi}\chi(x)$, where $\chi\in C_0^\infty$ equals $1$ in a neighborhood   $K$. Then for $x$ in a neighborhood of  $K$, $Pf(x) = e^{\i x\cdot\xi}(p_m(x,\xi)+r(x,\xi))$  with $r\in S^{m-1}$.  We have $|\xi|^m/C\le \|f\|_{H^m}\le C|\xi|^m$ for $|\xi|\ge1$. Therefore, for such $\xi$,
\[
C_1\|Pf\|_{L^2}/\|f\|_{H^m}\ge \|p_m(\cdot,\xi)/|\xi|^m\|_{L^2}  - C_2/|\xi|.
\]
Take the limit $|\xi|\to\infty$ to complete the proof. 
\end{proof}
We complete the proof of Theorem~\ref{thm_LR_stab} with the aid of Lemma~\ref{lemma_P}. We recover first the $L^2$-norms w.r.t.\ $x$ of $\mathcal{L}(x,\xi)$ uniformly in  $\xi$ (in fixed coordinates); we can choose $\mu=1$ then. Using standard interpolation estimates, we can estimate the $L^\infty$ norm of $\mathcal{L}(x,\xi)- \tilde{\mathcal{L}}(x,\xi)$ with $\mu<1$ in \r{QF}, using the a priori bounds on $g$ and $\tilde g$ in some $C^k$, $k\gg1$, which imply similar bounds on $\mathcal{L}$ and $\tilde{\mathcal{L}}$.
\end{proof}

\begin{remark}\label{rem2}
The symbol $\lambda_0$ can be computed. Since we do not use this formula, we will sketch the proof only. Using Green's formula, as in the proof of \cite[Prop.~2.1]{SU-JFA}, we can show that $2N_\mathcal{V}\cong D^*\Lambda$, where $\cong$ stands for equality  modulo lower order terms, and $N_\mathcal{V}$ is $N$ above with the subscript $\mathcal{V}$ indicating that it acts microlocally in that set.  %Then $4N_\mathcal{V}^*N_\mathcal{V}= \Lambda^*DD^*\Lambda$. 
The same proof implies that $\Lambda^*$ is the DN map associated with the incoming solution, i.e., the one which starts from $\mathcal{U}$ microlocally and hist $\mathcal{U}$. Therefore, $\Lambda^*\cong 2N_\mathcal{U}D^{-1}$, where $N_\mathcal{U}$ now acts in $\mathcal{U}$. Those two identities and the Egorov's theorem imply  $\lambda_0= -4(\xi_n \circ\mathcal L)\xi_n$, where $\xi_n$ is the function defined in \r{N0}. %If we know $g$ on $U$ and $V$, this symbol is a known 
\end{remark}

%Green's formula:
%\[
%(N_1f,f) \cong (N_2Df,Df)
%\]
%\[
%N_1\cong D^*N_2D = D^*\Lambda
%\]

\section{Applications and Examples} \label{sec_ex}

We start with a partial but still general enough case. 
We follow \cite[\S 24.1]{Hormander3}. Let $M$ be a Lorentzian manifolds with a timelike boundary $\bo$. Assume that $t$ is a real valued smooth function on $M$ so that the level surfaces $t=\text{const.}$ are compact and spacelike. For every $a<b$,   the (compact) ``cylinder'' $M_{ab}=\{a\le t\le b\}$ (assuming $[a,b]$ is in the range of $t$) has a boundary consisting of the spacelike surfaces $t^{-1}(a)$, $t^{-1}(b)$ and $\bo\cap M_{ab}$ which intersect transversely. This is a generalization of $[0,T]\times\Omega$  in the Riemannian case. By \cite[Theorem~24.1.1]{Hormander3}, the following problem is well posed
\[
Pu=0\; \text{in $M$}, \quad u|_{t<0}=0,\quad  u|_{\bo}=f 
\]
with $f\in H^s(\bo)$, $s\ge1$, $f=0$ for $t<0$;  with a unique solution   $u\in H^s(M)$ vanishing for $t<0$. Moreover, the map $f\mapsto u$ is continuous. Then the Dirichlet-to-Neumann map $\Lambda_{g,A,q}$ defined as in \r{DN}, is well defined. 

Let $x_0\in U_0\Subset U\subset\bo$ be as in Theorem~\ref{bdrystability}. Let $\chi$ be a properly supported \PDO\ cutoff of order zero localizing near some timelike covector over $x_0\in U_0$. Since there is a globally defined time function, there are no periodic lightlike geodesics. Then $\chi\Lambda_{g,A,q}\chi$ can be taken as $\Lambda_{g,A,q}^{\rm loc}$ and Theorem~\ref{bdrystability} applies. If we know a priori that $\Lambda_{g,A,q}: H_{(0)}^1(\bo)\to L^2(\bo)$ is continuous, where the subscript $(0)$ indicates functions vanishing for $t=0$, then we can replace $\Lambda_{g,A,q}^{\rm loc}$ by $\Lambda_{g,A,q}$ in \r{delta} and therefore, in Theorem~\ref{bdrystability}.

Similarly, with suitable \PDO\ cutoffs $\chi_1$ and $\chi_2$, we can take $\Lambda_{g,A,q}^{\rm loc} = \chi_1 \Lambda_{g,A,q}\chi_2$, under the assumptions of Theorem~\ref{intstability}. And again, if we know that $\Lambda_{g,A,q}: H_{(0)}^1(\bo)\to L^2(\bo)$ is continuous, we can remove the cutoffs. The results with the cutoffs are actually stronger.

Some special subcases are discussed below. They recover and extend the uniqueness results in 
\cite{MR1004174,Ramm-Sj, Ramm_Rakesh_91,waters2014stable,  Salazar_13, Aicha_15, Bellassoued_BA_16}, and some of the stability results there. %We skip some details. 
 Using the results in this paper with the support theorems about the light ray transform  in \cite{S-support2014, Siamak2016}, we can get new partial data results.

\begin{example}
Let $q$ be a unknown  potential but assume that the metric and the magnetic fields are known. Restrict the DN map to $M_{ab}$ for some $a<b$. Then we can recover $L_0q$ in a stable way as in Theorem~\ref{bdrystability} over all timelike geodesics intersecting the lateral boundary transversely at their endpoints. If $g$ is real-analytic, then we can apply the results in \cite{S-support2014} to recover $q$ in the set covered by those geodesics under an additional foliation condition. Note that in contrast, the results in \cite{Eskin2015} require $A$ and $q$ to be analytic in time. 
\end{example}

\begin{example}\label{ex2}
In the example above, assume that $g$ is Minkowski, and $M_{ab}= [0,T]\times\bar\Omega$ for  some bounded smooth $\Omega\subset\R^n$. By Theorem~\ref{bdrystability}, we can recover $L_1A$ and $L_0q$ over all lightlike geodesics (lines) $l_{z,\theta} = \{(t,x) = (s,z+s\theta);\; s\in\R\}$, $(z,\theta)\in \R^n\times S^{n-1}$, not intersecting the top and the bottom of the cylinder. By \cite{S-support2014}, we can recover $q$ in the set covered by those lines. By \cite{Siamak2016}, we can recover $A$ up to $\d\phi$, $\phi=0$ on $[0,T]\times\bo$ in that set as well. 

For example, if $\Omega$ is the ball $B(0,1)= \{x;\;|x|<1\}$, the DN map recovers uniquely $q$ and $A$, up to a gauge transform, in the cylinder  $[0,T]\times \bar B(0,1)$ with the upward characteristic cone with base $\{0\}\times B(0,1)$ and the downward with base $\{T\}\times B(0,1)$ removed, see Figure~\ref{DN_Lorentz_fig2}. If $T\le 2$, those two cones intersect; otherwise they do not but the result holds in both cases. This is the possibly reachable region from $[0,T]\times\bo$, thus the results are sharp since no information about the complement can be obtained by the finite speed of propagation. 

\begin{figure}[h!] % float placement: (h)ere, page (t)op, page (b)ottom, other (p)age
  \centering
  \includegraphics[trim = 20mm 130mm 100mm 35mm, clip, scale=0.5]{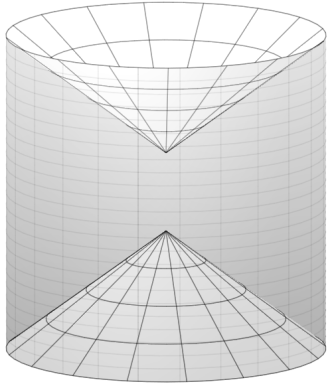}
  %trim={<left> <lower> <right> <upper>}
\caption{\small The DN map, with $g$ Minkowski, on the lateral boundary of the cylinder determines a potential and a magnetic field up to $\d\phi$ inside the cylinder but outside  the two characteristic cones.}
\label{DN_Lorentz_fig2}
\end{figure}

This extends further the uniqueness part of the results in \cite{MR1004174,Ramm-Sj, Ramm_Rakesh_91,waters2014stable,  Salazar_13, Aicha_15, Bellassoued_BA_16}. Using the stability estimate in \cite{Begmatov01} about $L_0$, and the logarithmic estimate for $L_1$ in \cite{Salazar2014}, we can use Theorem~\ref{intstability} to recover the results in \cite{Salazar2014}. One important improvement however is that for uniqueness,   we do not assume that $A$ and $q$ are known outside $[0,T]$; or that $T=\infty$ because the uniqueness results in \cite{S-support2014, Siamak2016} do not require the function or the vector field to be compactly supported in time. 
\end{example}

\begin{example}
A partial data case of Example~\ref{ex2} is the following. Let $\Gamma\subset \bo$ be relatively open, and assume that $\bo$ is strictly convex.  Assume that we know the DN map for $f$ supported in $[0,T]\times\Gamma$, and we measure $\Lambda f$ there, as well. Then we can recover $q$ (for all $n\ge2$) and $A$ for $n\ge3$, up to a gauge transform, in the set covered by the lightlike lines hitting $[0,T]\times\bo$ in $[0,T]\times\Gamma$ at their both endpoints. When $n=2$, the recovery of $A$ up to a potential $\d\phi$ requires that if we know $L_1A$ for all some lightlike $l_{z,\theta}$, we also know it for $l_{z,-\theta}$, see  \cite{Siamak2016}, and this is the reason we restricted $n$ to $n\ge3$. Those local uniqueness results for the DN maps are new. 
\end{example}

\begin{example}
In a recent work \cite{Bellassoued_BA_16}, the authors study the inverse problem for the wave operator
\[
P := \partial_t^2 +a(t,x)\partial_t - \Delta+ b(t,x)
\]
with real valued $a$, $b$. The coefficient $b$ causes absorption. We do not restrict $A$ and $q$ to be real valued, so we can take $A=(\frac\i2 a(t,x),0,\dots,0)$ , $q=-\frac\i2\partial_t a(t,x)+b(t,x)$, then $P$ in \r{P} is the same as the one above. Then Theorem~\ref{intstability} proves unique recovery of $A$, $q$ up to the gauge transform $A\mapsto A-\d\psi$ with $\psi=0$ on $[0,T]\times\partial\Omega$. Since $A$ is restricted to the class of covector fields with spatial components zero, we must have $\psi=\psi(t)$. However, then $\psi=0$ for $x\in\partial\Omega$ implies $\psi\equiv0$. Therefore, the logarithmic and the double logarithmic stability estimates in  \cite{Bellassoued_BA_16} for $a$ and for $b$ which are about the DN map can be obtained by Theorem~\ref{intstability} combined with the stability estimates in  \cite{Begmatov01, Salazar2014}. We can get new uniqueness results however with partial data as in the previous examples. In the Riemannian case studied by Montalto \cite{Carlos_12} we can allow an absorption term as well to obtain, up to a gauge transform, stable recovery of a Riemannian simple metric in a generic class, a magnetic field, a potential and an absorption term from the DN map. 
\end{example}

%\newpage
\appendix

\section{A Linear Algebra Lemma}

In the appendix we prove a linear algebra lemma which is used several times in the proof above. Let $V$ be an $(n+1)$-dimensional vector space, $g$ be a Lorentzian metric on $V$. Fix a basis of $V$, the following coordinates are with respect to this basis. Let $h$ be a bilinear form on $V$, which can be identified with a matrix $(h^{jk})_{j,k=0}^{n}$. Let $A$ be a linear functional on $V$, which can be identified with a row vector $(A^0,\dots,A^n)$. Define a linear function $L$ by
$$L(\xi):=h(\xi,\xi)+A(\xi)= h^{jk}\xi_{j}\xi_{k} + A^j\xi_j, \quad\quad \xi=(\xi_0,\dots,\xi_n)\in V.$$
\begin{lemma}
For any open subset $W$ of $\{\xi\in V: g(\xi,\xi)=-1\}$, there exists $N:=\frac{(n+1)(n+2)}{2}+(n+1)$ vectors $\xi^{(1)},\dots,\xi^{(N)}$ such that
$$\sum^{n}_{j,k=0}|h^{jk}|^2 + \sum^{n}_{j=0}|A^j|^2 \leq C \max_{\xi\in \{\xi^{(1)},\dots,\xi^{(N)}\}}L(\xi)$$
for some constant $C>0$. 
\end{lemma}
\begin{proof}
For each $\xi\in V$, we may regard $L(\xi)$ as a linear functional of $(h,A)$:
$$
L(\xi): (h,A) \mapsto h^{jk}\xi_j\xi_k + A^j\xi_j.
$$
The pair $(h,A)$ belongs to a linear space that can be identified with $\mathbb{R}^N$, thus $L(\xi)$ belongs to $(\mathbb{R}^N)'\equiv\mathbb{R}^N$. As $\{L(\xi):\xi\in W\}$ is a complete set, there exists a basis $L(\xi^{(1)}), \dots, L(\xi^{(N)})$ in it. This basis of linear functionals forms an invertible linear transformation $\mathbf{L}$ on $\mathbb{R}^N$:
$$
\mathbf{L}: (h,A) \mapsto (L(\xi^{(1)}), \dots, L(\xi^{(N)})).
$$
Inverting $\mathbf{L}$ yields the estimate.
\end{proof}

Now we are ready to prove a similar lemma for general Lorentzian manifolds. 
\begin{lemma} \label{mainlemma}
Let $(M,g)$ be an $(n+1)$-dimensional Lorentzian manifolds with continuous metric. Fix $x_0\in M$ and local coordinates $(x^0,\dots,x^n)$ near $x_0$. Let $h=h^{jk}(x)\frac{\partial}{\partial x^j}\otimes\frac{\partial}{\partial x^k}$ be a continuous symmetric 2-tensor field; let $A=A^j(x)\frac{\partial}{\partial x^j}$ be a continuous vector field. Define a functional on $T_xM$:
$$
L_x(\xi): (h(x),A(x)) \mapsto h^{jk}(x)\xi_j\xi_k + A^j(x)\xi_j, \quad\quad \xi\in T^{\ast}_{x}M.
$$
Then for any open subset $W$ of $\{\xi:g_x(\xi,\xi)=-1\}\subset T^{\ast}M$, there exists a neighborhood $U$ of $x_0$ in $M$ and $N:=\frac{(n+1)(n+2)}{2}+(n+1)$ codirections $\xi^{(1)}(x), \dots, \xi^{(N)}(x)\in W$ such that
\begin{equation} \label{algebra}
\sum^{n}_{j,k=0}|h^{jk}(x)|^2 + \sum^{n}_{j=0}|A^j(x)|^2 \leq C \max_{\xi(x)\in \{\xi^{(1)}(x),\dots,\xi^{(N)}(x)\}}L_x(\xi(x)), \quad\quad \text{ for all } x\in U.
\end{equation}
Here $C>0$ is a positive constant independent of $x$.
\end{lemma}
\begin{proof}
For any $x\in M$, the proof of the previous lemma shows there exist $\xi^{(1)}(x), \dots, \xi^{(N)}(x)$ in $T_xM\cap W$ such that $L(\xi^{(1)}(x)), \dots, L(\xi^{(1)}(x))$ is a basis of linear functionals on $\mathbb{R}^N$, and the estimate \eqref{algebra} holds at $x$. In particular this is true at $x=x_0$. Since each $L(\xi^{(k)}(x))$ depends on $x$ continuously, we conclude that the linear transformation
$$\mathbf{L}_x: (h(x),A(x)) \mapsto (L_x(\xi^{(1)}(x)), \dots, L_x(\xi^{(N)}(x)))$$
is continuously invertible in a neighborhood $U$ of $x_0$. Shrinking $U$ if necessary, we may assume the closure $\overline{U}$ is compact. Finally taking $C=\sup_{x\in\overline{U}}\|\mathbf{L}^{-1}\|$ completes the proof.
\end{proof}

\bibliographystyle{plain}
%\bibliography{myreferences}

\begin{thebibliography}{10}

\bibitem{BaoZhang}
G.~Bao and H.~Zhang.
\newblock Sensitivity analysis of an inverse problem for the wave equation with
  caustics.
\newblock {\em J. Amer. Math. Soc.}, 27(4):953--981, 2014.

\bibitem{Begmatov01}
A.~Begmatov.
\newblock A certain inversion problem for the ray transform with incomplete
  data.
\newblock {\em Siberian Math. Journal}, 42(3):428--434, 2001.

\bibitem{Belishev_87}
M.~I. Belishev.
\newblock An approach to multidimensional inverse problems for the wave
  equation.
\newblock {\em Dokl. Akad. Nauk SSSR}, 297(3):524--527, 1987.

\bibitem{belishev_2007}
M.~I. Belishev.
\newblock Recent progress in the boundary control method.
\newblock {\em Inverse Problems}, 23(5):R1--R67, 2007.

\bibitem{BelishevK92}
M.~I. Belishev and Y.~V. Kurylev.
\newblock To the reconstruction of a {R}iemannian manifold via its spectral
  data ({BC}-method).
\newblock {\em Comm. Partial Differential Equations}, 17(5-6):767--804, 1992.
\newblock MR1177292.

\bibitem{Bellassoued_BA_16}
M.~Bellassoued and I.~Ben~Aicha.
\newblock Stable determination outside a cloaking region of two time-dependent
  coefficients in an hyperbolic equation from {D}irichlet to {N}eumann map.
\newblock {\em arXiv:1605.03466}, 2016.

\bibitem{BellassouedDSF}
M.~Bellassoued and D.~Dos Santos~Ferreira.
\newblock Stability estimates for the anisotropic wave equation from the
  {D}irichlet-to-{N}eumann map.
\newblock {\em Inverse Probl. Imaging}, 5(4):745--773, 2011.

\bibitem{Aicha_15}
I.~Ben~Aicha.
\newblock Stability estimate for hyperbolic inverse problem with time dependent
  coefficient.
\newblock Inverse Problems, 31(12), 125010, 2015.

\bibitem{BQ}
J.~Boman and E.~T. Quinto.
\newblock Support theorems for {R}adon transforms on real analytic line
  complexes in three-space.
\newblock {\em Trans. Amer. Math. Soc.}, 335(2):877--890, 1993.

\bibitem{CooperS84}
J.~Cooper and W.~Strauss.
\newblock The leading singularity of a wave reflected by a moving boundary.
\newblock {\em J. Differential Equations}, 52(2):175--203, 1984.

\bibitem{DosSantosKSU09}
D.~Dos Santos~Ferreira, C.~E. Kenig, M.~Salo, and G.~Uhlmann.
\newblock Limiting {C}arleman weights and anisotropic inverse problems.
\newblock {\em Invent. Math.}, 178(1):119--171, 2009.

\bibitem{Eskin2015}
G.~Eskin.
\newblock Inverse problems for general second order hyperbolic equations with
  time-dependent coefficients.
\newblock 2015.

\bibitem{EskinR10}
G.~Eskin and J.~Ralston.
\newblock The determination of moving boundaries for hyperbolic equations.
\newblock {\em Inverse Problems}, 26(1):015001, 13, 2010.

\bibitem{Greenleaf-Uhlmann}
A.~Greenleaf and G.~Uhlmann.
\newblock Nonlocal inversion formulas for the {X}-ray transform.
\newblock {\em Duke Math. J.}, 58(1):205--240, 1989.

\bibitem{Greenleaf_Uhlmann90}
A.~Greenleaf and G.~Uhlmann.
\newblock Composition of some singular {F}ourier integral operators and
  estimates for restricted {X}-ray transforms.
\newblock {\em Ann. Inst. Fourier (Grenoble)}, 40(2):443--466, 1990.

\bibitem{Greenleaf_UhlmannCM}
A.~Greenleaf and G.~Uhlmann.
\newblock Microlocal techniques in integral geometry.
\newblock In {\em Integral geometry and tomography ({A}rcata, {CA}, 1989)},
  volume 113 of {\em Contemp. Math.}, pages 121--135. Amer. Math. Soc.,
  Providence, RI, 1990.

\bibitem{Hormander3}
L.~H{\"o}rmander.
\newblock {\em The analysis of linear partial differential operators. {III}},
  volume 274.
\newblock Springer-Verlag, Berlin, 1985.
\newblock Pseudodifferential operators.

\bibitem{Hormander4}
L.~H{\"o}rmander.
\newblock {\em The analysis of linear partial differential operators. {IV}},
  volume 275.
\newblock Springer-Verlag, Berlin, 1985.
\newblock Fourier integral operators.

\bibitem{IsakovSun92}
V.~Isakov and Z.~Q. Sun.
\newblock Stability estimates for hyperbolic inverse problems with local
  boundary data.
\newblock {\em Inverse Problems}, 8(2):193--206, 1992.

\bibitem{Kian2016}
Y.~Kian.
\newblock Recovery of time-dependent damping coefficients and potentials
  appearing in wave equations from partial data.
\newblock {\em arXiv:1603.09600}, 2016.

\bibitem{Kian_Oksanen_16}
Yavar Kian and Lauri Oksanen.
\newblock Recovery of time-dependent coefficient on riemanian manifold for
  hyperbolic equations.
\newblock {\em arXiv:1606.07243}, 2016.

\bibitem{KLU-Einstein-I}
Y.~Kurylev, M.~Lassas, and G.~Uhlmann.
\newblock Inverse problems in spacetime {I}: Inverse problems for {E}instein
  equations - {E}xtended preprint version.
\newblock {\em arXiv:1405.4503}, 2014.

\bibitem{KLU-2015}
Y.~Kurylev, M.~Lassas, and G.~Uhlmann.
\newblock Seeing through spacetime.
\newblock {\em arXiv:1405.3386}, 2015.

\bibitem{LOSU-strings}
M.~Lassas, L.~Oksanen, P.~Stefanov, and G.~Uhlmann.
\newblock On the inverse problem of finding cosmic strings and other
  topological defects.
\newblock to appear in  {\em Commun. Math. Phys.}, 2014.

\bibitem{LassasUW_2016}
M.~Lassas, G.~Uhlmann, and Y.~Wang.
\newblock Inverse problems for semilinear wave equations on lorentzian
  manifolds.
\newblock {\em arXiv preprint arXiv:1606.06261}, 2016.

\bibitem{Carlos_12}
C.~Montalto.
\newblock Stable determination of a simple metric, a covector field and a
  potential from the hyperbolic {D}irichlet-to-{N}eumann map.
\newblock {\em Comm. Partial Differential Equations}, 39(1):120--145, 2014.

\bibitem{Palais60}
R.~S. Palais.
\newblock Extending diffeomorphisms.
\newblock {\em Proc. Amer. Math. Soc.}, 11:274--277, 1960.

\bibitem{Petrov_book}
A.~Z. Petrov.
\newblock {\em Einstein {S}paces}.
\newblock Translated from the Russian by R. F. Kelleher. Translation edited by
  J. Woodrow. Pergamon Press, Oxford-Edinburgh-New York, 1969.

\bibitem{Siamak2016}
S.~RabieniaHaratbar.
\newblock Support theorem for the {L}ight {R}ay transform on {M}inkoswki
  spaces.
\newblock 2016.

\bibitem{Ramm_Rakesh_91}
A.~G. Ramm and Rakesh.
\newblock Property {$C$} and an inverse problem for a hyperbolic equation.
\newblock {\em J. Math. Anal. Appl.}, 156(1):209--219, 1991.

\bibitem{Ramm-Sj}
A.~G. Ramm and J.~Sj{\"o}strand.
\newblock An inverse problem of the wave equation.
\newblock {\em Math. Z.}, 206(1):119--130, 1991.

\bibitem{Salazar_13}
R.~Salazar.
\newblock Determination of time-dependent coefficients for a hyperbolic inverse
  problem.
\newblock {\em Inverse Problems}, 29(9):095015, 17, 2013.

\bibitem{Salazar2014}
R.~Salazar.
\newblock Stability estimate for the relativistic {S}chr\"odinger equation with
  time-dependent vector potentials.
\newblock {\em Inverse Problems}, 30(10):105005, 18, 2014.

\bibitem{MR1004174}
P.~Stefanov.
\newblock Uniqueness of the multi-dimensional inverse scattering problem for
  time dependent potentials.
\newblock {\em Math. Z.}, 201(4):541--559, 1989.

\bibitem{S-support2014}
P.~Stefanov.
\newblock Support theorems for the light ray transform on analytic lorentzian
  manifolds.
\newblock {\em arXiv:1504.01184, to appear in Proc. Amer. Math. Soc.}, 2015.

\bibitem{SU-JFA}
P.~Stefanov and G.~Uhlmann.
\newblock Stability estimates for the hyperbolic {D}irichlet to {N}eumann map
  in anisotropic media.
\newblock {\em J. Funct. Anal.}, 154(2):330--358, 1998.

\bibitem{SU-JAMS}
P.~Stefanov and G.~Uhlmann.
\newblock Boundary rigidity and stability for generic simple metrics.
\newblock {\em J. Amer. Math. Soc.}, 18(4):975--1003, 2005.

\bibitem{SU-IMRN}
P.~Stefanov and G.~Uhlmann.
\newblock Stable determination of generic simple metrics from the hyperbolic
  {D}irichlet-to-{N}eumann map.
\newblock {\em Int. Math. Res. Not.}, 17(17):1047--1061, 2005.

\bibitem{SUV-DNmap2014}
P.~Stefanov, G.~Uhlmann, and A.~Vasy.
\newblock On the stable recovery of a metric from the hyperbolic {DN} map with
  incomplete data.
\newblock {\em Inverse Problems and Imaging}, arXiv:1505.02853, 2016.

\bibitem{MR1115177}
P.~D. Stefanov.
\newblock Inverse scattering problem for moving obstacles.
\newblock {\em Math. Z.}, 207(3):461--480, 1991.

\bibitem{Sun_90}
Z.~Q. Sun.
\newblock On continuous dependence for an inverse initial-boundary value
  problem for the wave equation.
\newblock {\em J. Math. Anal. Appl.}, 150(1):188--204, 1990.

\bibitem{tataru95}
D.~Tataru.
\newblock Unique continuation for solutions to {PDE}'s; between {H}\"ormander's
  theorem and {H}olmgren's theorem.
\newblock {\em Comm. Partial Differential Equations}, 20(5-6):855--884, 1995.

\bibitem{Tataru99}
D.~Tataru.
\newblock Unique continuation for operators with partially analytic
  coefficients.
\newblock {\em J. Math. Pures Appl. (9)}, 78(5):505--521, 1999.

\bibitem{Taylor-book0}
M.~E. Taylor.
\newblock {\em Pseudodifferential operators}, volume~34 of {\em Princeton
  Mathematical Series}.
\newblock Princeton University Press, Princeton, N.J., 1981.

\bibitem{Taylor-book2}
M.~E. Taylor.
\newblock {\em Partial differential equations. {II}}, volume 116 of {\em
  Applied Mathematical Sciences}.
\newblock Springer-Verlag, New York, 1996.
\newblock Qualitative studies of linear equations.

\bibitem{waters2014stable}
A.~Waters.
\newblock Stable determination of {X}-{R}ay transforms of time dependent
  potentials from partial boundary data.
\newblock {\em Comm. Partial Differential Equations}, 39(12):2169--2197, 2014.

\end{thebibliography}
%\end{document}

\end{document}